\documentclass{amsart}
\usepackage{amssymb,amsmath,amsthm}
\usepackage{xypic}
\usepackage{graphicx}

\newtheorem{prop}{Proposition}
\newtheorem{lem}[prop]{Lemma}

\newtheorem{thm}[prop]{Theorem}

\theoremstyle{definition}

\newtheorem{defi}[prop]{Definition}
\newtheorem{rem}[prop]{Remark}

\newcommand{\fb}{\mathfrak{b}}

\newcommand{\fg}{\mathfrak{g}}
\newcommand{\fh}{\mathfrak{h}}
\newcommand{\fl}{\mathfrak{l}}

\newcommand{\fn}{\mathfrak{n}}

\newcommand{\fp}{\mathfrak{p}}

\newcommand{\fs}{\mathfrak{s}}

\newcommand{\calA}{\mathcal{A}}

\newcommand{\calC}{\mathcal{C}}
\newcommand{\calD}{\mathcal{D}}
\newcommand{\calE}{\mathcal{E}}
\newcommand{\calF}{\mathcal{F}}

\newcommand{\calH}{\mathcal{H}}
\newcommand{\calI}{\mathcal{I}}

\newcommand{\calL}{\mathcal{L}}
\newcommand{\calM}{\mathcal{M}}
\newcommand{\calN}{\mathcal{N}}
\newcommand{\calO}{\mathcal{O}}

\newcommand{\calR}{\mathcal{R}}

\newcommand{\calU}{\mathcal{U}}
\newcommand{\calV}{\mathcal{V}}

\newcommand{\bC}{\mathbb{C}}
\newcommand{\bD}{\mathbb{D}}

\newcommand{\bG}{\mathbb{G}}

\newcommand{\bL}{\mathbb{L}}

\newcommand{\bN}{\mathbb{N}}
\newcommand{\bO}{\mathbb{O}}
\newcommand{\bP}{\mathbb{P}}

\newcommand{\bR}{\mathbb{R}}

\newcommand{\bZ}{\mathbb{Z}}

\newcommand{\Lie}{{\rm Lie}}

\newcommand{\Rep}{{\rm Rep}}
\newcommand{\QCoh}{{\rm QCoh}}

\newcommand{\modu}{{\textrm -}{\rm mod}}

\newcommand{\itop}{{\it op}}

\newcommand{\conda}{(\ast)}

\DeclareMathOperator{\Supp}{Supp}

\DeclareMathOperator{\Ann}{Ann}

\DeclareMathOperator{\Sym}{Sym}
\DeclareMathOperator{\gr}{gr}
\DeclareMathOperator{\id}{id}

\DeclareMathOperator{\Gr}{Gr}
\DeclareMathOperator{\DR}{DR}

\DeclareMathOperator{\calHom}{\calH\!{\it om}}
\DeclareMathOperator{\Hom}{Hom}
\DeclareMathOperator{\totimes}{\overset{\#}{\otimes}}

\title[Radon transforms of twisted D-modules on partial flag varieties]{Radon transforms of twisted D-modules \\ on partial flag varieties}
\author{Kohei Yahiro}
\date{}
\address{Graduate School of Mathematical Sciences, The University of Tokyo, 3-8-1 Komaba Meguro-ku Tokyo 153-8914, Japan}
\email{yahiro@ms.u-tokyo.ac.jp}


\begin{document}
\setlength{\baselineskip}{13pt}

\begin{abstract}
In this paper we study intertwining functors for twisted D-modules on partial flag varieties and their relation to the representations of semisimple Lie algebras. We show that certain intertwining functors give equivalences of derived categories of twisted D-modules. This is a generalization of a result by Marastoni. We also show that these intertwining functors from dominant to antidominant direction are compatible with taking global sections. 
\end{abstract}

\maketitle

\section{Introduction}

In this paper we study the integral transforms for modules over sheaves of twisted differential operators on partial flag varieties which is called intertwining functors or Radon transforms and its relation to the representations of reductive Lie algebras over $\bC$. A sheaf of twisted differential operators (TDO) on a smooth algebraic variety is a sheaf of rings which is locally isomorphic to the sheaf of the differential operators. We call modules over a TDO twisted D-modules. Taking global sections induces a functor from the category of twisted D-modules on partial flag varieties $G/P$ to a category of representations of Lie algebra $\fg:=\Lie\ G$. Beilinson and Bernstein \cite{BeiBer81} established an equivalence of these categories. In \cite{BeiBer83} they defined intertwining functors for twisted D-modules on full flag varieties $G/B$. Intertwining functors are defined as integral transforms of twisted D-modules along the orbits of product of two flag varieties $G/B \times G/B$ and hence parametrized by the elements of the Weyl group. Intertwining functors change the parameter of TDO by an action of Weyl group. Beilinson and Bernstein proved that these intertwining functors are equivalences of categories. Marastoni \cite{Mar13} considered the integral transform between a partial flag variety $G/P$ and its opposite partial flag variety $G/P^{\itop}$ of D-modules and proved that it is an equivalence of derived categories. We extend the definition of intertwining functors to a certain class of orbits of the product of two partial flag varieties $G/P \times G/P'$ where $P$ and $P'$ are associate parabolic subgroups and prove that they give equivalences between derived categories of twisted D-modules (Theorem \ref{main1}). Mili\v ci\'c \cite{Mil94} studied the compatibility between intertwining functors and global section functors and proved that intertwining functors in one direction are compatible with global section functors. We extend his result to the intertwining functors defined in this paper (Theorem \ref{main2}). 

Let us now explain the preceding results, related results and our results in more details. 

Let $G$ be a connected reductive algebraic group over $\bC$, $B$ be its Borel subgroup and $H$ be a Cartan subgroup contained in $B$. We denote their Lie algebras by $\fg$, $\fb$ and $\fh$. We denote by $\Pi$ the set of simple roots and by $\rho$ the half sum of positive roots. 
We denote the enveloping algebra of $\fg$ by $\calU(\fg)$. Let $\lambda \in \fh^*$. We define the Verma modules by $M(\lambda):=\calU(\fg) \otimes_{\calU(\fb)} \bC_\lambda$, where $\bC_\lambda$ is regarded as a $\fb$-module by $\fb \to \fh$. We denote by $I(\lambda):=\Ann_{\calU(\fg)}(M(\lambda))$ the annihilator of the Verma module. 

The theorem of Beilinson and Bernstein relates representations of semisimple Lie algebras and D-modules on flag varieties. To state their result in full generality and to explain the results of this paper, we need the notion of sheaves of twisted differential operators (TDO). 
For the precise definition of the TDO, see Definition \ref{defTDO}. The isomorphism classes of TDO's on the flag variety $G/B$ are parametrized by the elements of $\fh^*$. For each $\lambda \in \fh^*$, there is a natural way to construct a corresponding TDO $\calD^\lambda_{G/B}$ and a homomorphism $\psi^\lambda: \calU(\fg) \to \Gamma(G/B, \calD^\lambda_{G/B})$. We denote by $\calD^\lambda_{G/B}\modu$ the category of quasi-coherent $\calD^\lambda_{G/B}$-modules. The localization theorem of Beilinson and Bernstein (\cite{BeiBer81}) states the following. 
The homomorphism of algebras $\psi^\lambda: \calU(\fg) \to \Gamma(G/B, \calD^\lambda_{G/B})$ factors through an isomorphism $\calU(\fg)/I(\lambda-2\rho) \cong \Gamma(G/B, \calD^{\lambda}_{G/B})$ and if $\lambda$ is regular and dominant the functor $\Gamma: \calD^\lambda_{G/B}\modu \to \calU(\fg)/I(\lambda-2\rho)\modu$ which assign to a $\calD^{\lambda}_{G/B}$-module $\calM$ the space $\Gamma(G/B, \calM)$ of all global sections is an equivalence of categories. For the definition of regularity and dominance, see Definition \ref{regdom}. 
Note that our choice of positive roots is the opposite to that of Beilinson and Bernstein. 
The inverse functor $\Delta^\lambda$ (see \S \ref{TDOflag}) is called the localization functor. 

This theorem connects the representation theory of semisimple Lie algebras and the geometry of the flag variety. For example, the results of Kazhdan and Lusztig \cite{KazLus80} and Lusztig and Vogan \cite{LusVog83} on the perverse sheaves on flag varieties can be applied via the localization theorem and the Riemann-Hilbert correspondence (\cite[Chapter VIII]{Bor87}) to the representation theory and yield a formula of multiplicities of standard modules, one of which is known as the Kazhdan-Lusztig conjecture. 

Beilinson and Bernstein in \cite{BeiBer83} studied the $\calD^\lambda_{G/B}$-module for not necessarily antidominant $\lambda$. In this case the functor $\Gamma$ is not exact. But they proved that the localization theorem still holds for regular $\lambda$ if we consider derived categories \cite[\S 13. Corollary]{BeiBer83}. 

Backelin and Kremnizer studied the case of non-regular $\lambda$ and established a localization theorem \cite{BacKre15} using the relative enveloping algebra of Borho and Brylinski \cite{BorBry89}. 

An analogue of the localization theorem still holds for the partial flag variety $G/P_I$, where $P_I$ is a parabolic subgroup of $G$ which contains $B$ corresponding to $I \subset \Pi$. 
Isomorphism classes of TDO's on the partial flag variety $G/P_I$ are parametrized by $(\fh/\fh_I)^*$, where $\fh_I$ is the subalgebra generated by the coroots $\check\alpha$, $\alpha \in I$. For partial flag varieties the homomorphism $\psi^\lambda_I: \calU(\fg) \to \Gamma(G/P_I, \calD^\lambda_{G/P_I})$ is not always a surjection. For a regular and antidominant weight $\lambda \in (\fh/\fh_I)^* \subset \fh^*$, the following result is known. The homomorphism $\psi^\lambda_I$ is surjective and the functor $\Gamma: \calD^\lambda_{G/P_I}\modu \to \Gamma(G/P_I, \calD^\lambda_{G/P_I})\modu$ is an equivalence of categories. 
This theorem is stated in \cite{BeiBer81} and a proof is found in \cite[Theorem 6.3]{Bie90}. In Proposition \ref{propequiv} we show that this theorem still holds for any regular weight $\lambda \in (\fh/\fh_I)^*$ if we consider derived categories. 
Bien used the localization theorem for dominant weight on partial flag varieties to study discrete spectrum of the semisimple symmetric space. Kitchen studied the relation of the global section functor on $G/B$ and that on $G/P_I$ under the pullback along the quotient map $G/B \to G/P_I$ and proved that the functor $\Gamma$ commutes with the pullback \cite[Theorem 5.1]{Kit12}. She used this result to study the global sectons of standard twisted D-modules on partial flag varieties.

Beilinson and Bernstein defined an intertwining functor for full flag varieties $G/B$ in \cite[\S 11]{BeiBer83}. The intertwining functors are defined as integral transforms of twisted D-modules along the $G$-orbit under the diagonal $G$-action on $G/B \times G/B$. Thus intertwining functors are parametrized by elements $w \in W$ of the Weyl group and changes $\lambda$ by the action of the Weyl group $w(\lambda-\rho) +\rho$, in a way that $\Gamma(G/B, \calD^\lambda_{G/B})$ are unchanged. 
Beilinson and Bernstein proved that the intertwining functors are equivalences of derived categories. 
They used intertwining functors to prove the Casselman's submodule theorem \cite[Theorem 1]{BeiBer83}. 
Mili\v ci\'c \cite{Mil94} studied the property of intertwining functors and proved that an intertwining functor 
in one direction commutes with the derived functor of the global section functor. In this paper we generalize this result to partial flag varieties (Main Theorem \ref{main2}). This is one of the main results of this paper. The result by Mili\v ci\'c is used to give a classification of irreducible admissible $(\fg, K)$-modules. Kashiwara and Tanisaki \cite{KasTan96} studied the case of affine flag varieties. They showed that intertwining functors are equivalences of categories and that an intertwining functor 
in one direction commutes with the derived functor $\bR\Gamma$. They used these results to prove the Kazhdan-Lusztig conjecture for affine flag varieties. 
Marastoni studied the Radon transform of (non-twisted) D-modules on Grassmannian varieties \cite[Theorem 1]{Mar98} and general partial flag varieties \cite[Theorem 1.1]{Mar13} in the case intertwining functor is given by the open orbit in $G/P \times G/P^{\textit{op}}$, where $P^{\textit{op}}$ is the opposite of $P$ in $G$. We generalize his result to intertwining functors given by more general orbits (Main Theorem \ref{main1}). This is also one of the main results of this paper. 

Intertwining functors are studied from different perspectives. We mention some of related results. 
D'Agnolo and Schapira \cite{DAgSch96a} established general theory of integral transform of D-modules along a correspondence. In \cite{DAgSch96b} they applied their theory for the $n$-dimensional projective space $\bP$ and the dual projective space $\bP^*$ with the correspondence given by the closed orbit of the product $\bP \times \bP^*$ under the diagonal action of the general linear group $GL(n+1)$.
Marastoni and Tanisaki \cite{MarTan03} treated the Radon transform for two partial flag varieties when the Radon transform is given by the closed $G$-orbit. They studied how weakly equivariant D-modules behave under the Radon transform. 

Yun \cite{Yun09} studied integral transforms of perverse sheaves which are constructible along fixed stratifications. 
If the stratifications on both sides satisfies some good properties with respect to the correspondence, he proved that the Radon transform with respect to the correspondence is an equivalence of derived categories and that the Radon transform sends tilting objects to projective objects. 
The stratifications of $G/P$ and $G/P^{\textit{op}}$ with respect to $B$-orbits and the open $G$-orbit of $G/P \times G/P^{\textit{op}}$ satisfy the assumptions of Yun's theorem and he obtained a category equivalence. This equivalence is a special case of Marastoni's result in the sense that the categories of sheaves constructible along these strata are the category of D-modules that are smooth along $B$-orbits by the Riemann-Hilbert correspondence. The method of Yun has an advantage that it allows to calculate the weights of mixed perverse sheaves. Yun's theorem is also applicable to the Radon transform between an affine flag variety and its opposite thick flag variety. 

Arkhipov and Gaitsgory \cite{ArkGai??} studied the intertwining operators for the category of twisted D-modules on an affine flag variety and its opposite thick affine flag variety using D-modules on the moduli stack of principal $G$-bundles on $\bP^1$ with reductions to the Borel subgroup at 0 and $\infty$, which can be regarded as the quotient stack $G \backslash (G/I \times G/I^{\itop})$ for the algebraic loop group. 

Cautis, Dodd and Kamnitzer \cite{CauDodKamX16} constructed categorical $\fs\fl_2$ action on $\bigoplus_{0 \leq i \leq n} D^b(\calD_{\Gr(i,n),h}\modu)$, the direct sum of derived categories of filtered D-modules on Grassmannian varieties. 
They showed that the resulting equivalence of category $D^b(\calD_{\Gr(i,n),h}\modu) \cong D^b(\calD_{\Gr(n-i,n),h}\modu)$ is given by the Radon transform along the open $GL(n)$-orbit of the product. 

Let us now explain the results in this paper. 

Let $I$ and $J$ be subsets of the set of simple roots $\Pi$ of $G$. We have corresponding parabolic subgroups $P_I$ and $P_J$ of $G$. The $G$-orbits of $G/P_J \times G/P_I$ are parametrized by double cosets in $W_I\backslash W/W_J$ of the Weyl group by the parabolic subgroups $W_I$ and $W_J$. We denote by $\bO_w$ the orbit corresponding to $w$. It is possible to define an integral transform for any $G$-orbit on the product, but to consider twisted D-modules we restrict to the case of $w$ for which the projections from $\bO_w$ to $G/P_I$ and $G/P_J$ are affine space fibrations, i.e., $w$ for which $wJ=I$ holds (Condition $\conda$).

We define the intertwining functors $R^{w,\mu}_{+}$ and $R^{w,\mu}_{!}$ for $w \in W$ and $\mu \in X^*(P_I)$ by first pulling back the twisted D-modules from $G/P_I$ to $\bO_w$, then tensoring by the invertible sheaf $\calL^{\mu} \otimes \det(\Theta_{p^w_1})$, and then pushing it forward to $G/P_J$ (Definition \ref{defradon}). Here $\det\Theta_{p^w_1}$ is the determinant invertible sheaf of relative tangent sheaf of the projection $p^w_1:\bO_w \to G/P_J$ and $\calL^\mu$ the $G$-equivariant invertible sheaf associated to $\mu$. 
The intertwining functors $R^{w, \mu}_+$ and $R^{w, \mu}_!$ send $D^b(\calD^\lambda_{G/P_I}\modu)$ to $D^b(\calD^{w^{-1}(\lambda-\rho)+\rho+w^{-1}\mu}_{G/P_J}\modu)$. 
The first main result of this paper is that the intertwining functors for these $w$ give equivalences of derived categories. 

\begin{thm}[{Theorem \ref{thmequiv}}]\label{main1}~\\
The functors $R^{w,\mu}_+$ and $R^{w^{-1},-w^{-1}\mu}_!$ are mutually inverse equivalences. 
\end{thm}
If we set $\lambda =0$, $\mu=\rho - w\rho$ and $w$ to be the minimal coset representative of longest element of $W$, this theorem specializes to the result of Marastoni \cite[Theorem 1.1]{Mar13}. 

Next we consider the compatibility of the intertwining functor for $\mu=0$ and the global section functors. We denote by $R^{w}_+$ and $R^{w}_!$ the intertwining functors for $\mu=0$. 

We denote by $\bR\Gamma^\lambda_I$ the composition of the derived functor of taking global section $D^b(\calD^\lambda_{G/P_I}\modu) \to D^b(\Gamma(G/P_I, \calD^\lambda_{G/P_I})\modu)$ and the pullback along ${\rm U}^\lambda_I:=\calU(\fg)/{\rm Ker}(\psi^\lambda_I) \to \Gamma(G/P_I, \calD^\lambda_{G/P_I})$.
We have natural morphisms of functors $I^w_+: \bR\Gamma^{\lambda}_I \to \bR\Gamma^{w^{-1}*\lambda}_J \circ R^w_{+}$ and $I^w_!:\bR\Gamma^{\lambda}_I \circ R^w_! \to \bR\Gamma^{w^{-1}*\lambda}_J$ (Proposition \ref{intmor}). 
We give a sufficient condition for $I^w_+$, $I^w_!$ to be isomorphisms. 
We need some notation. We define $v[\alpha, I] \in W$ for $\alpha \in \Pi \setminus I$ by $v[\alpha, I]= w^{I \cup \{\alpha\}}_0w^{I}_0$, where $w^I_0$ is the longest element of $W_I$. 
Take $\alpha_1, \ldots,  \alpha_r$ in Proposition \ref{BH} and let $I_0= I = v[\alpha_1, I_1]I_1,I_1 = v[\alpha_2, I_2]I_2, \ldots , I_{r-1} = v[\alpha_r, I_r]I_r , I_r=J$. We define the (scalar) generalized Verma module by $M^\fg_{\fp_I}(\mu):= \calU(\fg)\otimes_{\calU(\fp_I)}\bC_\mu$ for a character $\mu$ of $\fp_I$. For $K_1 \subset K_2 \subset \Pi$, we denote by $\fl_{K_1}$ the Levi subalgebra of $\fg$ corresponding to $K_1$ containing $\fh$ and by $\fp^{K_2}_{K_1}$ the parabolic subalgebra $\fl_{K_2}\cap \fp_{K_1}$ of $\fl_{K_2}$. 

The second main result of this paper is the following. 
\begin{thm}[{Theorem \ref{thmisom}}]\label{main2}
Let $\lambda_0=\lambda \in (\fh/\fh_I)^*$ and $\lambda_i := v[\alpha_i, I_i]^{-1}*\lambda_{i-1}$. 
Assume that $\lambda$ is regular and for each $i$ the generalized Verma module $M^{\fl_{I_i \cup \{\alpha_i\}}}_{\fp^{I_i \cup \{\alpha_i\}}_{I_i}}({v[\alpha_i,I_i]^{-1}\lambda_{i-1}})$ of the Levi subalgebra is irreducible. Then the morphisms $I^w_+: \bR\Gamma^{\lambda}_I \to \bR\Gamma^{w^{-1}*\lambda}_J \circ R^w_{+}$ and $I^w_!:\bR\Gamma^{\lambda}_I \circ R^w_! \to \bR\Gamma^{w^{-1}*\lambda}_J$ are isomorphisms of functors. 
\end{thm} 
The generalized Verma modules appearing in this theorem are tensor products of generalized Verma modules induced from a maximal parabolic subalgebra and a one dimensional representation. It is irreducible if ${v[\alpha_i,I_i]^{-1} \lambda_{i-1}}$ is antidominant. A criterion of the irreducibility of is given by Jantzen \cite{Jan77}. He, Kubo and Zierau gave a complete list of reducible parameters for scalar generalized Verma modules associated to maximal parabolic subalgebras of simple Lie algebras \cite{HeKubZie??}.  
For complete flag varieties $G/B$, this theorem coincides with the result of Mili\v ci\'c \cite[Theorem L.3.23]{Mil94}

Let us briefly describe the outline of this paper. 
In subsection \ref{sectTDO} we recall the general properties of sheaves of twisted differential operators on smooth algebraic varieties. 
In subsection \ref{sectflag} we recall basic facts on partial flag varieties and representations of semisimple Lie algebras which are needed in this paper. 
In section \ref{sectradon} we define intertwining functors (Radon transforms) for a class of orbits in product of partial flag varieties and prove that they are equivalences of derived categories (Theorem \ref{thmequiv}). 
In section \ref{sectsect} we study the compatibility of global section functors and intertwining functors. We prove a localization theorem (Proposition \ref{propequiv}) and use this to prove the compatibility of global section functors and intertwining functors from dominant to antidominant direction (Theorem \ref{thmisom}). 

The author wishes to express his gratitude to his advisor Hisayosi Matumoto for introducing this subject to the author. The author also thanks him for his encouragement and advice and indicating the proof of Lemma \ref{lemsurj}. The author thanks Syu Kato and Yoichi Mieda for reading this paper and pointing out many typos and mistakes. 
This work is partially supported by Grant-in-Aid for JSPS Fellows (No. 12J09386).

\section{Preliminary}

\subsection{Notation}

We always work over the field $\bC$ of complex numbers. 

For a ring $A$, we denote by $A\modu$ the category of left $A$-modules. 
For a morphism of rings $f:A \to B$, we denote by $f^*$ the pullback functor $B\modu \to A\modu$. 

For  algebraic groups $G$, $B$, $P_I$, \ldots , we denote their Lie algebras by $\fg$, $\fb$, $\fp_I$, \ldots .
We denote by $\Rep(G)$ the category of rational representations of $G$. We denote by $X^*(G)$ the group of characters of $G$. 
For a character $\lambda$ of $G$ or $\fg$, we denote by $\bC_\lambda$ the corresponding one dimensional representation.

We always denote by id the identity functor on a category. For an abelian category $\calC$, we denote by $D^b(\calC)$ the bounded derived category of $\calC$ and by $D^-(\calC)$ the derived category consisting of bounded above complexes. 

Let $f$ be a continuous map between topological spaces. We denote by $f^{-1}$ the pullback of sheaves and by $f_*$ the pushforward of sheaves. We denote by $f_!$ the proper pushforward. For a sheaf $\calF$ on a topological space $X$, we denote by $\Gamma(\calF)$ the set of all sections of $\calF$ on $X$ instead of $\Gamma(X, \calF)$. 

Let $f: X \to Y$ be a morphism of algebraic varieties $X$ and $Y$. We denote by $f^*$ the pullback which is defined by $f^*(\calM)= \calO_X \otimes_{f^{-1}(\calO_Y)} f^{-1}(\calM)$ for an $\calO_Y$-module $\calM$ and  by $f_*$ the pushforward. 

We denote by $\{\star\}$ the reduced algebraic variety consisting of only one point and by $\star$ its point. For an  algebraic variety $X$, we denote by $a_X$ the unique morphism from $X$ to $\{\star\}$. For a locally free $\calO_X$-modules $\calV$, we denote by $\det(\calV)$ the determinant invertible sheaf. 
For a smooth algebraic variety $X$,  $\Theta_{X}$ is its tangent sheaf, $\Omega_X$ is its cotangent sheaf and $T^*X$ is the cotangent bundle. 
Let $f: X \to Y$ be a morphism of smooth algebraic varieties. Then we denote by $\omega_f$ the relative canonical sheaf of $f$.
Let $f: X \to Y$ be a smooth surjective morphism of smooth algebraic varieties $X$ and $Y$. We denote by $\Theta_f$ the relative tangent sheaf and by $\Omega_f$ the relative cotangent sheaf.

\subsection{Sheaves of twisted differential operators}\label{sectTDO}

In this subsection we recall the definition and properties of sheaves of twisted differential operators following Kashiwara and Tanisaki \cite{Kas89, KasTan96}. Note that in \cite{KasTan96} they use right modules while we use left modules and our notation is different from theirs.

\subsubsection{Definition of sheaves of twisted differential operators}

Let $X$ be a smooth algebraic variety. We denote by $\calO_X$ the sheaf of regular functions on $X$ and by $\calD_X$ the sheaf of differential operators on $X$. 

\begin{defi}\label{defTDO}
A sheaf of rings $\calA$ on X with a homomorphism $\iota: \calO_X \to \calA$ and an increasing filtration $(F^m\!\calA)_{m \in \bN}$ of $\calA$ by coherent $\calO_X$-submodules are called a \emph{sheaf of twisted differential operators (TDO)} on $X$ if following properties hold. 

1) The homomorphism $\iota$ induces an isomorphism $\calO_X \cong F^0\!\calA$.

2) $F^{m_1}\!\calA \cdot F^{m_2}\!\calA \subset F^{m_1+m_2}\!\calA$ 

3) [$F^{m_1}\!\calA , F^{m_2}\!\calA ] \subset F^{m_1+m_2-1}\!\calA$

\noindent The property 3) allows us to define a homomorphism of $\calO_X$-modules $\sigma: \gr_F^1\calA \to \Theta_X$ by defining ${\rm gr}_F^1\calA \ni \bar{a} \mapsto ( f \mapsto [a, f] ) \in \Theta_X$. 

4) $\sigma$ is an isomorphism. 

5) $\Sym^{\bullet}\Theta_X \to \gr_F^{\bullet}\calA$ induced by $\sigma^{-1}$ is an isomorphism. 

\end{defi}

For a coherent $\calA$-module, we define its characteristic variety $Ch(\calM)$ which is a closed conic subset of the cotangent space $ T^*X$ using good filtrations in the same way for D-modules. 
A coherent $\calA$-module is called \emph{holonomic} if its characteristic variety is a Lagrangian subvariety of $T^*X$. 
We denote by $\calA\modu$ the category of quasi-coherent $\calA$-modules, by $D^b(\calA\modu)$ its bounded derived category and by $D^b_{{\textit hol}}(\calA\modu)$ full subcategory consisting of complexes whose cohomology in each degree is holonomic.

There is a natural bijection between the set of all isomorphism classes of TDO's and $H^2(X, \sigma^{\geq 1}\Omega^{\bullet}_X)$ \cite[Theorem 2.6.1]{Kas89}, where $\Omega^{\bullet}_X$ is the de Rham complex of $X$ and $\sigma^{\geq 1}$ is the brutal truncation, i.e. replacing the degree $\leq 0$ term of the complex by 0. Denote the cohomology class corresponding to $\calA$ under this bijection by $c(\calA) \in H^2(X, \sigma^{\geq 1}\Omega^{\bullet}_X)$. 

For each $x \in X$, we have an $\calA$-module $\calA \otimes_{\calO_X} \bC_x$, where $\bC_x$ is the skyscraper sheaf supported on $x$ with $1$-dimensional fiber, which has a canonical structure of an $\calO_X$-module. This is a holonomic $\calA$-module. We denote this $\calA$-module by $\calA(x)$. 

\subsubsection{Operations on sheaves of twisted differential operators and on their modules}

Let $X$ and $Y$ be a smooth algebraic variety, $f:X \to Y$ a morphism, $\calA$, $\calA_1$, $\calA_2$ be TDO's on $Y$, and $\calL$ be an invertible sheaf on $Y$. 
We denote by $c_1(\calL) \in H^2(Y; \bC)$ the first Chern class of $\calL$ defined below. We have a homomorphism of abelian groups ${\rm dlog}: \calO^*_X \to \Omega^1_X $ define by $f \to f^{-1}df$. The homomorphism ${\rm dlog}$ induces a homomorphism $H^1(X,\calO^*_X) \to H^1(X, \Omega^1_X )$. We define $c_1(\calL)$ as the image of the class of $\calL$ in $H^1(X,\calO^*_X)$ under the composition of ${\rm dlog}$ with $H^1(X, \Omega^1_X ) \to H^2(X, \sigma^{\geq 1}\Omega^{\bullet}_X) $

First we recall operations on TDO's. 

\begin{defi}
We denote by $\calA^{op}$ the opposite ring of $\calA$. 
\end{defi}
We have $c(\calA^{op})=-c(\calA)+c_1(\Omega_Y)$ (\cite[\S 2.7.1]{Kas89}). 

\begin{defi}[{after the first Remark\ 2.6.5 \cite{Kas89}}]~\\
Let $a$ be a complex number. There is a TDO $\calA^{\calL^a}$ with the property $c(\calA^{\calL^a})=c(\calA)+a c_1(\calL)$
\end{defi}
When $a$ is an integer, then $\calD^{\calL^a}$ is the sheaf of differential operators $\calE\!\textit{nd}^{\textit{fin}}_{\bC}(\calL^a)$ acting on $\calL^a$ defined below. 

Let $\calR$ be a sheaf of rings on $Y$ and $\calM$ be a left $\calO_Y$- right $\calR$-module. 
We define a sheaf of filtered rings $\calE\!\textit{nd}^{\textit{fin}}_{\calR}(\calM) $ as follows. 
First we define $F^0\calE\!\textit{nd}^{\textit{fin}}_{\calR}(\calM)$ to be the image of homomorphism $\calO_X \to \calE\!\textit{nd}_{\calR}(\calM)$. 
We define $F^n\calE\!\textit{nd}^{\textit{fin}}_{\calR}(\calM)$ for $n \in \bN$ recursively by $F^n\calE\!\textit{nd}^{\textit{fin}}_{\calR}(\calM):= \{ r \in F^{n}\calE\!\textit{nd}^{\textit{fin}}_{\calR}(\calM) \mid [\calO_X, r] \subset F^{n-1}\calE\!\textit{nd}^{fin}_{\calR}(\calM) \}$. 
Finally we define a sheaf of rings $\calE\!\textit{nd}^{\textit{fin}}_{\calR}(\calM)$ by $\bigcup_{i \in \bN} F^i\calE\!\textit{nd}^{\textit{fin}}_{\calR}(\calM)$.

\begin{defi}
We define a TDO $\calA_1 \# \calA_2$ by $\calE\!\textit{nd}^{\textit{fin}}_{\calA_1 \otimes_{\bC} \calA_2}(\calA_1 \otimes_{\calO_Y} \calA_2)$, where the tensor product is taken using left $\calO_X$-module structures of $\calA_1$ and $\calA_2$. 
\end{defi}
We have $c(\calA_1 \# \calA_2)= c(\calA_1)+c(\calA_2)$ (\cite[Lemma 1.1.1]{KasTan96}).

\begin{defi}
We define a TDO
$\calA^{-\#}$ by $(\calA^{{\textit op}})^{\Omega_Y^{-1}}$. 
\end{defi}

We have an isomorphism of TDO's $\calA \# \calA^{-\#} \cong \calD_X$.

\begin{defi}
We define a TDO $f^{\#}\!\calA$ on X by $\calE\!\textit{nd}^{\textit{fin}}_{f^{-1}(\calA)}(f^*(\calA))$
\end{defi}

\begin{prop}[{\cite[Lemma 1.1.5]{KasTan96}}] We have following isomorphisms of TDO's on X. 
\begin{enumerate}
\item $f^{\#}\calD_Y \cong \calD_X$
\item $f^{\#}(\calA_1 \# \calA_2) \cong f^{\#}\!\calA_1 \# f^{\#}\!\calA_2$
\item $f^{\#}\!\calA^{\calL^{a}} \cong (f^{\#}\!\calA)^{(f^{*}\!\calL)^{a}}$
\end{enumerate}
\end{prop}

Next we recall operations on modules over TDO's. 

\begin{defi}~\\
(1) Let $\calN_1 \in D^b_{{\textit hol}}(\calA_1\modu)$ and $\calN_2 \in D^b_{{\textit hol}}(\calA_2\modu)$. 
We say that $\calN_1$ and $\calN_2$ are \emph{non-characteristic} if ${\rm Ch}(\calN_1) \cap {\rm Ch}(\calN_2) \subset T^*_YY$. 
~\\
(2) Let $\calN \in D^b(\calA\modu)$. We say that $\calN$ is \emph{non-characteristic with respect to $f$} if the inclusion
$(X \times_Y Ch(\calN)) \cap T^*_XY \subset X \times_Y T^*_YY$ holds, where $T^*_XY := {\rm Ker} (X \times_Y T^*Y \to T^*X)$. 

\end{defi}

\begin{prop}
The tensor product $\otimes_{\calO_Y}$ induces a functor
\[\totimes : D^b(\calA_1\modu) \times D^b(\calA_2\modu) \to D^b(\calA_1 \# \calA_2\modu).\]  
This functor sends complexes with holonomic cohomologies to that with with holonomic cohomologies. 
\end{prop}

\begin{defi}
We define the duality functor $\bD : D^b_{{\textit hol}}(\calA\modu) \to D^b_{{\textit hol}}(\calA^{-\#}\modu)$ by assigning $\calM \mapsto \bR\calHom_{\calA}(\calM, \calA) \otimes \omega^{-1}_Y[\dim Y]$. 
\end{defi}

The following propositions state basic properties of the duality functor.

\begin{prop}[{\cite[Proposition 1.2.1]{KasTan96}}]\label{doubledual}~\\
We have an isomorphism of functors
$\bD \circ \bD \cong \id$ on ${D^b_{{\textit hol}}(\calA\modu)}$. 
\end{prop}
 
\begin{prop}[{\cite[Proposition 1.2.2]{KasTan96}}]~\\
Assume that $\calN_1 \in D^b_{{\textit hol}}(\calA_1\modu)$ and $\calN_2 \in D^b_{{\textit hol}}(\calA_2\modu)$ are non-characteristic. Then we have an isomorphism \[\bD(\calN_1) \totimes \bD(\calN_2) \cong \bD(\calN_1 \totimes \calN_2). \]
\end{prop}

\begin{defi}(Pullback)

We define the functor $f^!: D^b(\calA\modu) \to D^b(f^{\#}\!\calA\modu)$ by 
\[\calM \mapsto f^!\calM := f^{*}(\calA) \otimes^{\bL}_{f^{-1}(\calA)} f^{-1}(\calM).\]
This functor preserves holonomicity. 

We define the functor $f^+: D^b_{{\textit hol}}(\calA\modu) \to D^b_{{\textit hol}}(f^{\#}\!\calA\modu)$ by $f^+ := \bD \circ f^! \circ \bD$. 

\end{defi}

Note that we have a canonical isomorphism $f^!\calM \cong f^*\calM$ of $\calO_X$-modules. 

\begin{defi}(Pushforward)

We define the functor $f_+: D^b(f^{\#}\!\calA\modu) \to D^b(\calA\modu)$ by 
\[\calM \mapsto \bR f_*(( f^!(\calA^{\textit op}) \overset{\#}\otimes \omega_f) \otimes^{\bL}_{f^{-1}\calA} f^{-1}(\calM)).\]
This functor preserves holonomicity. 

We define the functor $f_+: D^b_{{\textit hol}}(f^{\#}\!\calA\modu) \to D^b_{{\textit hol}}(\calA\modu) $ by $f_+ := \bD \circ f_! \circ \bD$. 

\end{defi}

\begin{prop}[{\cite[Proposition 1.2.4]{KasTan96}}]\label{nc}~\\
(i) Let $\calN \in D^b(\calA\modu)$ be non-characteristic with respect to $f$. 
Then we have $f^+\calN \cong f^!\calN$. 

The non-characteristic assumption holds automatically if $f$ is smooth. 

\noindent(ii) There is a morphism of functors $f_! \to f_+$. 

For $\calM \in D^b_{hol}(f^{\#}\!\calA)$ such that $\Supp(\calM) \to Y$ is projective, the morphism of functor induces an isomorphism $f_!(\calM) \cong f_+(M)$. 

If $f$ is projective then the assumption holds automatically. 

\end{prop}

\begin{prop}[{Monoidal property and projection formula \cite[Proposition 1.2.5]{KasTan96}}]\label{pullmon}~\\
{\rm (i)} For $\calN_1 \in D^b_{{\textit hol}}(\calA_1\modu), \calN_2 \in D^b_{{\textit hol}}(\calA_2\modu)$, we have an isomorphism $f^! (\calN_1 \totimes \calN_2) \cong f^!(\calN_1) \totimes f^!(\calN_2)$. 

\noindent{\rm (ii)} For $\calM \in D^b_{{\textit hol}}(f^{\#}\!\calA_1\modu)$ and $\calN \in D^b_{{\textit hol}}(\calA_2\modu)$, we have $f_+(\calM \totimes f^!(\calN)) \cong (f_+(\calM) \totimes \calN)$. 
\end{prop}

\begin{prop}[{Base change isomorphism \cite[Proposition 1.2.6]{KasTan96}}]\label{bc}~\\
Let
\[
\xymatrix{
X^{\prime} \ar[r]^{f^{\prime}} \ar[d]^{g^{\prime}} &  Y^{\prime} \ar[d]^{g}  \\
X \ar[r]^{f} & Y
}
\]
be a cartesian diagram of smooth varieties. Then for $\calM \in D^b_{{\textit hol}}(g^{\#}\!\calA\modu)$, we have isomorphisms 
$ g^{\prime}_! (f^{\prime+}(\calM)) \cong f^{+}(g^{}_!(\calM))$, $ g^{\prime}_+ (f^{\prime!}(\calM)) \cong f^{!}(g^{}_+(\calM))$. 

\end{prop}

\subsubsection{Sheaves of twisted differential operators on homogeneous spaces}

Let $G$ be an algebraic group and $X$ be a smooth $G$-variety. We denote by $\mu$ the action $G \times X \to X$ and by $p$ the projection $G \times X \to X$. Recall that a quasi-coherent $\calO_X$-module $\calF$ with an isomorphism $\beta: \mu^*\calF \to p^*\calF$ is called \emph{$G$-equivariant} if $\beta$ satisfies the compatibility conditions (4.4.2) and (4.4.3) of \cite{Kas89}. We denote by $\QCoh_G(X)$ the category of $G$-equivariant quasi-coherent $\calO_X$-modules. 
A TDO $\calA$ with an isomorphism of TDO $\alpha : \mu^\#\!\calA \to p^\#\!\calA$ is called \emph{$G$-equivariant TDO} if the compatibility conditions (4.6.1) and (4.6.2) of \cite{Kas89} are satisfied. 
Let $\calA$ be a $G$-equivariant TDO. An $\calA$-module $\calM$ which is a $G$-equivariant quasi-coherent $\calO_X$-module with $\beta$ is \emph{weakly $G$-equivariant} if $\beta$ is a homomorphism of $p^*\!\calA$-modules.

Now let $X$ be a homogeneous $G$-variety. The action gives rise to a homomorphism of Lie algebras $\fg \to \Gamma(\Theta_X)$. 
Fix a point $x \in X$. Let $G_x$ be the stabilizer of $x$ in $G$ and $\fg_x$ be its Lie algebra. For a quasi-coherent sheaf $\calF$ on $X$, $\calF(x)$ denotes its fiber over $x$. We have the following equivalence of categories. 

\begin{prop}[{\cite[Theorem 4.8.1]{Kas89}}]\label{equivO}
The functor $\QCoh_G(X) \to \Rep(G_x)$ which sends $\calF \in \QCoh_G(X)$ to $\calF(x)$ is an equivalence of abelian categories. 
\end{prop}

We denote the inverse of this equivalence by $(\bullet)_X$. The invertible sheaf on $X$ associated to a character $\lambda$ of $G_x$ by this equivalence is denoted by $\calL^\lambda_X$.

The morphism of Lie algebras $\fg \to \Gamma(Theta_X)$ given by the action of $G$ on $X$ induces $\fg_X$ a structure of a Lie algebroid (for the definition of Lie algebroids, see \cite[\S 1.2]{BeiBer93}). We denote by $\calU(\fg_X)$ the enveloping algebra of the Lie algebroid $\fg_X$. The kernel of the structure map $\fg_X \to \Gamma(\Theta_X)$ is denoted by $\calI_X$. We have an isomorphism $\calI_X \cong (\fg_x)_X$ as Lie algebroids. 
Let $\lambda \in (\fg_x^*)^{G_x}$ be a $G_x$-invariant functional. We note that if $G_x$ is connected then $(\fg_x^*)^{G_x}$ is isomorphic to $(\fg_x/[\fg_x,\fg_x])^{*}$, the set of all characters of the Lie algebra $\fg_x$. The character $\lambda$ induces a character $\lambda_X :\calI_X \to \calO_X$. 

\begin{defi}
We define a sheaf of rings by
$\calD^\lambda_X:= \calU(\fg_X)/\langle A - \lambda_X(A) \mid A \in \calI_X \rangle$. 
\end{defi}
This is a $G$-equivariant TDO. We call $\calD^\lambda_X$ a \emph{$G$-equivariant TDO associated to $\lambda$}. If $\lambda$ comes from a character $\lambda$ of $G_x$, then we have an identity $c(\calD^\lambda_X) = c_1(\calL^\lambda_X)$ and hence an isomorphism of TDO's $\calD^\lambda_X \cong \calD^{\calL^\lambda_X}_X$. 

This construction is compatible with the pullback along a morphism of homogeneous spaces. 
\begin{prop}[{\cite[Proposition 4.14.1]{Kas89}}]
Let $\iota : H_1 \hookrightarrow H_2$ be closed subgroups of $G$. Let $p: G/H_1 \to G/H_2$ be the quotient morphism and $\lambda \in (\fh_2^*)^{H_2}$. Then we have an isomorphism of $G$-equivariant TDO's $p^{\#}\calD^{\lambda}_{G/H_2} \cong \calD^{d\iota^*\lambda}_{G/H_1}$. 
\end{prop}
In the following we suppress $d\iota^*$ from notation and write like $\calD^{\lambda}_{G/H_1} \cong p^{\#}\calD^{\lambda}_{G/H_2}$.

Fix $\lambda \in (\fg_x^*)^{G_x}$. A twisted $(\fg, G_x)$-module $M$ with the twist $\lambda$ is a $\fg$-module with a $G_x$-module structure on $\bC_\lambda \otimes M$ satisfying (4.10.1) and (4.10.2) of \cite{Kas89}. 
We have the following equivalence of categories. 

\begin{prop}[{\cite[Theorem 4.10.2 (1)]{Kas89}}]\label{wequivequiv}
The functor in Proposition \ref{equivO} induces an equivalence between the category of weakly equivariant $\calD^\lambda_X$-modules and the category of twisted $(\fg,G_x)$-modules with the twist $\lambda$. 
\end{prop}

%%%%%%%%%%%%%%%%%%%%%%%%%%%%%%%%%%%%%%%%%%%%%%%%%%%%%%%%%%%%%%%
\subsection{Partial flag varieties and TDO's on partial flag varieties}\label{sectflag}
The notation in this section is used throughout this paper. 

\subsubsection{Partial flag varieties}

Let $G$ be a connected reductive algebraic group over $\bC$, $B$ be its Borel subgroup, $U$ the unipotent radical of $B$ and $H$ be a maximal torus in $B$. 
We denote by $W$ the Weyl group $N_G(H)/H$, by $\Delta$ the set of roots of $\fg :={\rm Lie}\ G$, by $\Delta^+$ the set of positive roots determined by $B$ and by $\Pi$ the set of simple roots. We denote by $\ell$ the length function of $W$. 

To each subset $I \subset \Pi$, one associates a parabolic subgroup $P_I$ of $G$ in the way that $P_\emptyset = B$ holds, its Levi subgroup $L_{I}$ containing $H$, the unipotent radical $U_I$ of $P_I$, $H_I$ the subgroup of $H$ generated by the image of $\alpha :\bG_m \to H$ for all $\alpha \in I$, $\Delta_I$ the set of roots in $\fl_I$ and the parabolic subgroup $W_I$ of $W$. We denote by $w^{I}_0$ the longest element of $W_I$. We denote by $\bar P_I$ the opposite parabolic of $P_I$ and by $\bar U_I$ its unipotent radical. 
Let $I \subset J \subset \Pi$. We denote by $P^J_I$ the parabolic subgroup of $L_J$ defined by $L_J \cap P_I$. 
For $\alpha \in \Pi$ we denote by $\varpi_\alpha$ the fundamental weight corresponding to $\alpha$. 

We always identify $(\fh/\fh_I)^*$ with a subspace of $\fh^*$ via the natural inclusion and identify $X^*(B) \cong X^*(H)$ with a subgroup of $\fh^*$ and $X^*(P_I) \cong X^*(H/H_I)$ with a subgroup of $(\fh/\fh_I)^*$ via the differential. 

The partial flag variety $G/P_I$ decomposes into the finite union of $B$-orbits (Bruhat decomposition): we have $G/P_I= \coprod_{w \in W/W_I} BwP_I $. We denote the Bruhat cell $BwP_I$ by $C_w$. Each cell $C_w$ is an affine space with dimension the length of the minimal coset representative of $w$. We denote by $i_w$ the inclusion $C_w \hookrightarrow G/P_I$.  
Since $G/P_I$ is projective and has the Bruhat decomposition, by the Hodge theory we have following isomorphisms \cite[Theorem 5.5]{BerGelGel73}
\[H^2(G/P_I, \sigma^{\geq 1}\Omega^{\bullet}_{G/P_I}) \cong H^2(G/P_I; \bC) \cong (\fp_I / [\fp_I, \fp_I])^* \cong (\fh/\fh_I)^* .\] 
In the following we identify $(\fh/\fh_I)^*$ with $(\fp_I / [\fp_I, \fp_I])^*$. Note that the equality $c(\calD^{\lambda}_{G/P_I})=c_1(\calL^{\lambda}_{G/P_I})$ holds for any $\lambda \in X^*(P_I)$. 

\begin{rem}[{\cite[Theorem V]{Bot57}}]
The $G$-module $\Gamma(G/P_I, \calL^\lambda)$ is isomorphic to the finite dimensional irreducible $G$-module of lowest weight $\lambda$ or zero. 
\end{rem}

Let $I,J$ be subsets of $ \Pi$. 
The $G$-orbits of $G/P_J \times G/P_I$ are parametrized by the set $W_I \backslash W / W_J$. 
The correspondence is given by assigning to $w \in W_I \backslash W / W_J$ the orbit $\bO_w := G(w, e) \subset G/P_J \times G/P_I$.
Let $p^w_1: \bO_w \to G/P_J$ and $p^w_2: \bO_w \to G/P_I$ be restrictions of the first and the second projections from $G/P_J \times G/P_I$ and $j_w: \bO_w \to G/P_J \times G/P_I$ be the inclusion. 
The $G$-orbit $\bO_w$ is isomorphic to $G/(P_I \cap wP_Jw^{-1})$ as a $G$-variety. 
Under this isomorphism, $p^w_1: G/(P_I \cap wP_Jw^{-1}) \to G/P_I$ is the quotient morphism and $p^w_2 :G/(P_I \cap wP_Jw^{-1}) \to G/P_J$ is given by $g \mapsto gw$.  

In this paper we always consider $w \in W$ satisfying the following condition $\conda$. 
\[\textit{There exist}\  I, J \subset \Pi \ \textit{such that} \  wJ=I \ \textit{holds}.  \tag{$\ast$}\]
 For such $w$ we have $wL_Jw^{-1} = L_I$ and the morphism $p^w_1$ and $p^w_2$ are affine space fibrations with the fibers over identity cosets isomorphic to $P_J/(w^{-1}P_Iw \cap P_J) \cong U_J/(w^{-1}U_Iw \cap U_J)$ and $P_I/(P_I \cap wP_Jw^{-1}) \cong U_I/(U_I \cap wU_Jw^{-1})$ which are of dimension $\ell(w)$. From this fact we see that there is an isomorphism $\det(\Theta_{p^w_1}) \cong p^{w*}_2\calL^{w\rho-\rho}_{G/P_I}$. 

Let $w \in W$ satisfy Condition $\conda$. 
For such $w$ we have a ``reduced expression'' in the following sense.  
To each $\alpha \in \Pi \setminus I$ one associates $v[\alpha, I]=w^{I \cup \{ \alpha \}}_0w^I_0 \in W$. 
\begin{prop}[{\cite[Proposition 2.3]{BriHow99}}]\label{BH}
Let $I, J \subset \Pi$ and $w \in W$ satisfy $I=wJ$. Then there exist $\alpha_1, \ldots \alpha_r \in \Pi$ satisfying following conditions. 
\begin{enumerate}
\item $I = v[\alpha_1, I_1]I_1,I_1 = v[\alpha_2, I_2]I_2, \ldots , I_{r-1} = v[\alpha_r, I_r]I_r , I_r=J$
\item $\alpha_i \notin I_i$
\item $w=v[\alpha_1, I_1]\cdots v[\alpha_r, I_r]$
\item $\displaystyle \ell(w)=\sum_{1\leq i\leq r} \ell(v[\alpha_i, I_i])$
\end{enumerate}
\end{prop}
The element $v[\alpha, I] \in W$ may be thought of as a simple reflection in the parabolic case.

\subsubsection{Representations of semisimple Lie algebras}\label{repss}

Let $I \subset \Pi$. 
We denote by $\rho$ the half sum of positive roots of $\fg$, by $\rho_I$ the $\rho$ for $\fl_I$, by $\rho_{\fn_{I}}$ the difference $\rho -\rho_I$. 

\begin{defi}\label{regdom}
A weight $\lambda \in \fh^*$ is called \emph{regular} if $\langle \lambda-\rho, \check\alpha \rangle \neq 0$ holds for any root $\alpha \in \Delta$. 
A weight $\lambda \in \fh^*$ is called \emph{antidominant} if $\langle \lambda -\rho, \check\alpha \rangle \notin \bZ_{\geq 1}$ for all $\alpha \in \Delta^+$. 
\end{defi}
Note that the definition of regularity is different from usual one because we use $*$-action defined in Definition \ref{staraction}. 

We define the (scalar) generalized Verma module of highest weight $\lambda \in (\fh/\fh_I)^*$ by $M^{\fg}_{\fp_I}(\lambda) := \calU(\fg) \otimes_{\calU(\fp_I)} \bC_{\lambda}$. 
We denote by $I_{\fp_I}(\lambda)$ the annihilator of the generalized Verma module $M^{\fg}_{\fp_I}(\lambda)$. 
We denote by ${\rm U}^{\lambda}_I$ the quotient $\calU(\fg)/ I_{\fp_I}(\lambda - 2 \rho_{\fn_I})$. If $I$ is empty, we denote by ${\rm U}^{\lambda} $ the quotient $ \calU(\fg)/ I_{\fb}(\lambda - 2 \rho)$. 
We use the following result of Jantzen. 

\begin{prop}[{\cite[Corollar 15.27]{Jan83}}]\label{sameideal}
Assume that $J =w^{-1}I \subset \Pi$ holds. For any $\lambda \in (\fh/\fh_I)^*$, the ideals $I_{\fp_I}(\lambda)$ and $I_{\fp_J}(w^{-1}(\lambda + \rho) -\rho)$ coincide. 
\end{prop}

Let $V_1, V_2$ be $\fg$-modules. We define a $\fg$-bimodule $L(V_1, V_2)$ to be the $\fg$-subbimodule of $\Hom_\bC(V_1, V_2)$ consisting of all $\fg$-finite elements under the diagonal $\fg$-action. 

The homomorphism $\calU(\fg) \to {\rm End}_\bC(M^{\fg}_{\fp_I}(\lambda- 2\rho_{\fn_I}))$ factors through an homomorphism $\calU(\fg) \to L(M^{\fg}_{\fp_I}(\lambda- 2\rho_{\fn_I}), M^{\fg}_{\fp_I}(\lambda- 2\rho_{\fn_I}))$. 
This homomorphism factors through an injection $a^{\lambda}: {\rm U}^\lambda_I \to L(M^{\fg}_{\fp_I}(\lambda- 2\rho_{\fn_I}), M^{\fg}_{\fp_I}(\lambda- 2\rho_{\fn_I}))$. 
In general $a^\lambda$ is not surjective. An example of nonsurjectivity is given in \cite[\S 8.2]{Soe89}. 
For an ``antidominant regular'' weight $\lambda$, it is known that $a^\lambda$ is surjective. 

\begin{prop}[{\cite[Corollar 15.23]{Jan83}}]\label{asurj}
If $\lambda \in (\fh/\fh_I)^*$ satisfies $\langle \lambda + \rho , \check\beta \rangle \notin \bZ_{\geq 1}$ for all $\beta \in \Delta^+ \setminus \Delta_I$, then the homomorphism $a^\lambda$ is surjective. 
\end{prop}

\subsubsection{Sheaves of twisted differential operators on partial flag varieties}\label{TDOflag}

By the isomorphism $H^2(G/P_I, \sigma^{\geq 1} \Omega^{\bullet}_{G/P_I}) \cong (\fp_I / [\fp_I, \fp_I])^*$, we see that every TDO on partial flag varieties is a $G$-equivariant TDO.

We have a homomorphism of Lie algebras $\fg \to \Gamma(\calD^\lambda_{G/P_I})$ and an induced homomorphism of algebras $\psi^{\lambda}: \calU(\fg) \to \Gamma(\calD^\lambda_{G/P_I})$.

We first recall the fundamental result of Beilinson and Bernstein. 
Let $\lambda \in \fh^*$ be antidominant. 

\begin{prop}[{\cite[Lemme]{BeiBer81}}]
The homomorphism $\psi^\lambda$ induces an isomorphism ${\rm U}^\lambda \to \Gamma(\calD^\lambda_{G/B})$.
\end{prop}

\begin{thm}[{\cite[Th\'eor\`eme principal]{BeiBer81}}]\label{BDequiv}
Assume furthermore that $\lambda$ is regular. 
The functor $\Gamma: \calD^\lambda_{G/B}\modu \to {\rm U}^\lambda\modu$ which associates to a $\calD^\lambda_{G/B}$-module its global sections is an equivalence of categories. 
\end{thm}
This is the famous Beilinson-Bernstein localization theorem. An inverse to the functor $\Gamma$ is described as follows. 
Let $M$ be a ${\rm U}^\lambda$-module. To each open subset $V$ of $G/B$, we associate $\Gamma(V, \calD^\lambda_{G/B}) \otimes_{{\rm U}^\lambda} M$. The sheafification of this presheaf is a $\calD^\lambda_{G/B}$-module $\Delta^\lambda(M)$. This construction gives a functor $\Delta^\lambda: {\rm U}^\lambda\modu \to \calD^\lambda_{G/B}\modu$. 

If $\lambda$ is not antidominant, the exactness of the functor $\Gamma$ fails. In this case for regular $\lambda$ we have the following equivalence between derived categories due to Beilinson and Bernstein. 

\begin{thm}[{\cite[\S 13. Corollary]{BeiBer83}}]
Assume that $\lambda \in \fh^*$ is regular. 
The functor $\bR\Gamma: D^b(\calD^\lambda_{G/B}\modu) \to D^b({\rm U}^\lambda\modu)$ is an equivalence of categories. Its inverse is given by $\bL\Delta^\lambda$. 
\end{thm}

We now turn to the case of partial flag varieties.
Let $I$ be a subset of $\Pi$ and $\lambda \in (\fh/\fh_I)^*$. 
We first consider the general property of the global section functor. 
Taking global sections induces a functor $\Gamma : \calD^\lambda_{G/P_I}\modu \to \Gamma(\calD^\lambda_{G/P_I})\modu$. Let $\Delta_I$ be the localization functor $\calD^{\lambda}_{G/P_I}\otimes_{\Gamma(\calD^{\lambda}_{G/P_I})}(\bullet)$. 
The localization functor $\Delta_I$ is left adjoint to $\Gamma$, 
i.e. we have a functorial isomorphism $\Hom_{\calD^{\lambda}_{G/P_I}}(\Delta_I(N), \calM) \cong \Hom_{\Gamma(\calD^{\lambda}_{G/P_I})}(N, \Gamma(\calM))$ 
for $N \in \Gamma(\calD^\lambda_{G/P_I})\modu$ and $\calM \in \calD^\lambda_{G/P_I}\modu$. 
We denote its counit and unit by $\epsilon :\Delta_I \circ \Gamma \to \id$ and $\eta : \id \to \Gamma \circ \Delta_I$. 
We use the same symbols $\epsilon$ and $\eta$ for unit and counit for derived functors. 

The following theorem is stated in \cite{BeiBer81}. A proof is explained in \cite[Theorem 6.3]{Bie90}. 

\begin{prop}
Assume that $\lambda$ is regular and antidominant. Then the functor $\Gamma: \calD^\lambda_{G/{P_I}}\modu \to \Gamma(\calD^\lambda_{G/P_I})\modu$ is an equivalence of categories. 
\end{prop}

Next we recall properties of TDO $\calD^\lambda_{G/P_I}$.
The higher cohomology of TDO itself vanishes. 
\begin{prop}[{\cite[Lemma 1.4]{BorBry82}}]\label{Dcohvan}
For any $\lambda \in (\fh/\fh_I)^*$ and for any $i > 0$, we have an isomorphism 
$H^i(G/P_I, \calD^\lambda_{G/P_I}) \cong H^i(T^*G/P_I, \calO_{T^*G/P_I}) \cong 0$. 
\end{prop}
\noindent The identity coset $eP_I \in G/P_I$ is the unique closed $B$-orbit. 
The fiber of $\calD^\lambda_{G/P_I}$ at $eP_I$ is an irreducible $\calD^\lambda_{G/P_I}$-module supported on the point $eP_I$. 
The vector space of sections of $\calD^\lambda_{G/P_I}(eP_I)$ has a structure of a $\fg$-module through the homomorphism $\psi^{\lambda}: \calU(\fg) \to \Gamma(\calD^\lambda_{G/P_I})$. This $\fg$-module is isomorphic to a generalized Verma module. 

\begin{prop}[{\cite[Proposition 4]{Soe89}}]\label{fiberGVM}
The $\fg$-module $\Gamma(\calD^\lambda_{G/P_I}(eP_I))$ is isomorphic to the generalized Verma module $M^{\fg}_{\fp_I}(\lambda - 2\rho_{\fn_I})$. 
\end{prop}
Note that $\calD^\lambda_{G/P_I}(eP_I)$ is irreducible as a $\calD^\lambda_{G/P_I}$-module, but it is not necessarily irreducible as a $\fg$-module, even if $\psi^\lambda$ is surjective. 
Using this proposition, the kernel of $\psi^{\lambda}$ is described as follows. 
\begin{prop}[{\cite[Proposition 14]{Soe89}}]
The kernel of $\psi^\lambda$ coincides with $I_{\fp_I}(\lambda-2\rho_{\fn_I})$. 
\end{prop}
We denote the induced homomorphism ${\rm U}^\lambda_I \to \Gamma(\calD^\lambda_{G/P_I})$ also by $\psi^{\lambda}$.
We have a natural homomorphism of algebras $\Gamma (\calD^\lambda_{G/P_I}) \to {\rm End}_\bC\Gamma(\calD^\lambda_{G/P_I}(eP_I))$. 
By Proposition \ref{fiberGVM} we obtain a homomorphism of algebras $\Gamma (\calD^\lambda_{G/P_I}) \to {\rm End}_\bC(M^{\fg}_{\fp_I}(\lambda- 2\rho_{\fn_I}))$, which is $\fg$-equivariant with respect to the adjoint $\fg$-action on both sides. 
Since the adjoint $\fg$-action on $\Gamma (\calD^\lambda_{G/P_I})$ is locally finite, this homomorphism factors through the $\alpha^\lambda: \Gamma (\calD^\lambda_{G/P_I}) \to L(M^{\fg}_{\fp_I}(\lambda- 2\rho_{\fn_I}), M^{\fg}_{\fp_I}(\lambda- 2\rho_{\fn_I}))$. 

Soergel proved that this homomorphism is always an isomorphism. 

\begin{prop}[{\cite[Corollar 7]{Soe89}}]\label{alphasurj}
The homomorphism $\alpha^\lambda$ is an isomorphism. 
\end{prop}
By the construction we have $a^\lambda = \psi^\lambda \circ \alpha^\lambda$. 
This equality and the above proposition indicate that $a^\lambda$ is an isomorphism if and only if $\psi^\lambda$ is an isomorphism. 
Thus for $\lambda \in (\fh/\fh_I)^*$ satisfying the assumption of Proposition \ref{asurj}, $\psi^\lambda$ is an isomorphism. 
For some good parabolic subgroups, a stronger statement holds. 

\begin{prop}[{\cite{BorBry82}}]
If the moment map $T^*G/P_I \to \fg^*$ is birational onto the image and the image is normal, then $\alpha^\lambda$ is an isomorphism. 
\end{prop}
As a special case of this proposition, we have that $\psi^{\lambda}$ is isomorphism for full flag varieties. 
In Lemma \ref{lemsurj}, we prove that if $\lambda \in (\fh/\fh_I)^*$ is regular the morphism $\psi^\lambda$ is an isomorphism. 

Finally we state a result due to Kitchen, which states that taking pullbacks to flag variety is compatible with global sections. 
We denote by $p_I$ the quotient morphism $G/B \to G/P_I$. We have a pullback functor $p^!_I:D^b(\calD^\lambda_{G/B}\modu) \to D^b(\calD^\lambda_{G/P_I}\modu)$. Since $\alpha^\lambda: {\rm U}^\lambda \to \Gamma(\calD^\lambda_{G/B})$ is an isomorphism, the homomorphism $\psi^\lambda$ induces a homomorphism $q_I: \Gamma (\calD^\lambda_{G/B}) \to \Gamma (\calD^\lambda_{G/P_I})$.

\begin{prop}[{\cite[Corollary 5.2]{Kit12}}]\label{propKit}
We have an isomorphism of functors
$\bR\Gamma(G/B, -)\circ p^{!}_I \cong q^*_I \circ \bR\Gamma(G/P_I, -):D^{b}(\calD^{\lambda}_{G/P_I}\modu) \to D^b(\Gamma (\calD^\lambda_{G/B})\modu)$. 
\end{prop}

\[
\xymatrix{
D^{b}(\calD^{\lambda}_{G/P_I}\modu) \ar[rr]^-{\bR\Gamma(G/P_I, -)} \ar[d]_{p^!_I} \ar@{}[rrd]|{\circlearrowright} &  & D^b(\Gamma (\calD^\lambda_{G/P_I})\modu) \ar[d]^{q^*_I} \\
D^{b}(\calD^{\lambda}_{G/B}\modu) \ar[rr]_-{\bR\Gamma(G/B, -)} & & D^b(\Gamma (\calD^\lambda_{G/B})\modu)
}
\]

%%%%%%%%%%%%%%%%%%%%%%%%%%%%%%%%%%%%%%%%%%%%%%%%%%%%%%%%%%%%%%%
\section{Radon transforms for partial flag varieties}\label{sectradon}

We define an affine action of the Weyl group on $\fh^*$, which appears many times in this paper. 
 
\begin{defi}\label{staraction}
For $w \in W$ and $\lambda \in \fh^*$, we define $w*\lambda$ by $w*\lambda:= w(\lambda-\rho) +\rho$.
\end{defi}
Note that this action differs from the dot action which is defined in \cite[\S 2.3]{Jan83}.

Let $I \subset \Pi $. 
In this paper we consider only $w \in W$ satisfying $I=wJ$ for some $J \subset \Pi$. 
In this case $p^w_1$ and $p^w_2$ are affine space fibrations. This assumption has a following drawback. 

\begin{lem}\label{lemtwist}
Let $I, J \subset \Pi$ and $w \in W$ satisfy $wJ=I$. 
\begin{enumerate}
\item The pullback $p^{w*}_1: H^*(G/P_J, \bC) \to H^*(\bO_w ,\bC)$ and $p^{w*}_2: H^*(G/P_I, \bC) \to H^*(\bO_w ,\bC)$ are isomorphisms. 
\item Under the identification $H^2(G/P_I, \bC) \cong (\fh/\fh_I)^*$ and $H^2(G/P_J, \bC) \cong (\fh/\fh_J)^*$, the linear map $(p^{w*}_1)^{-1} \circ p^{w*}_2$ coincides with $w^{-1}$. 
\end{enumerate}
\end{lem} 

\begin{proof}
1. This follows from the fact that $p^w_1$ and $p^w_2$ are affine space fibration and hence have contractible fibers. 

\noindent 2. Pick $\lambda \in X^*(P_J) \subset (\fh/\fh_J)^*$. 
Then we have an isomorphism $p^{w*}_1\calL^\lambda_{G/P_J} \cong \calL^{w\lambda}_{\bO_w} \cong p^{w*}_2\calL^{w\lambda}_{G/P_I} $. Since $(\fh/\fh_J)^*$ is generated by $X^*(P_J)$ as a $\bC$-vector space, we have an equality $(p^{w*}_1)^{-1} \circ p^{w*}_2 = w^{-1}$. 
\end{proof}

We consider integral transforms arising from $G$-orbits $\bO_w$ of $G/P_J \times G/P_I$ for $w$ satisfying Condition $\conda$. 
\begin{defi}(Intertwining functor or Radon transform)\label{defradon}~\\
For each $w \in W$ satisfying $wJ=I$ and each $\mu \in X^*(P_I)$, we define the \emph{intertwining functor} or the \emph{Radon transform} $R^{w, \mu}_{?}$ for $?=!$ or $+$ associated to $w$ and $\mu$ by
\[
R^{w, \mu}_?(-):= p^w_{1?}(\det\Theta_{p^w_1} \totimes \calL^{\mu}_{\bO_w} \totimes p^{w!}_2 ( - )) : D^b_{\textit hol}(\calD^{\lambda}_{G/P_I}\modu) \to D^b_{\textit hol}(\calD^{w^{-1}*\lambda+w^{-1}\mu}_{G/P_J}\modu). 
\]
\end{defi}
The functor $R^{w,\mu}_!$ is also defined on the category $D^b(\calD^{\lambda}_{G/P_I}\modu)$. 
If $\mu=0$, we omit $\mu$ and denote by $R^{w}_?$. 

The previous lemma and the isomorphism $\det\Theta_{p^w_1} \cong \calL^{-\rho+w\rho}_{\bO_w}$ explain the twist in the codomain of the intertwining functor. 

Intertwining functors are given by kernels on the product $G/P_J \times G/P_I$. 
Let $j_w: \bO_w \hookrightarrow G/P_J \times G/P_I$ be the inclusion. 
We have the following description of the intertwining functor using a kernel. 

\begin{lem}\label{intker}
Let $\calM \in D^b(\calD^\lambda_{G/P_I}\modu)$. We have the following isomorphism for $?= !, *$. 
 \[R^w_?(\calM) \cong p_{1+}( j_{w?}(\det\Theta_{p_1} \totimes \calL^{\mu}_{\bO_w}) \totimes p^{!}_2(\calM))\]
\end{lem}

\begin{proof}
This follows immediately from the projection formula (Proposition \ref{pullmon} (ii)). 
\begin{align*}
R^w_?(\calM)  &\cong p^w_{1?}(\det\Theta_{p^w_1} \totimes \calL^{\mu}_{\bO_w} \totimes p^{w!}_2 ( \calM )) \\
 & \cong p_{1+}(j_{w?}(\det\Theta_{p^w_1} \totimes \calL^{\mu}_{\bO_w} ) \totimes p^{!}_2 (\calM))
\end{align*}
Here $j_{w?}$ is a functor $D^b(\calD^{-\rho+w\rho+\mu}_{\bO_w}\modu) \to D^b(p^{\#}_1\calD^{w^{-1}\lambda}_{G/P_J} \# p^\#_2\calD^{-\lambda -\rho +w\rho +\mu }_{G/P_I}\modu)$. 
Note that the both of $p_1$ and $p_2$ are smooth and proper morphisms. 
\end{proof}

\begin{defi}
Let $w \in W$ and $\lambda \in (\fh/\fh_I)^*$. 
We define the kernel of the intertwining functor by
\[K^{w,\mu}_?  := j_{w?}(\det\Theta_{p^w_1} \totimes \calL^{\mu}_{\bO_w})\in D^b(p^{\#}_1\calD^{w^{-1}\lambda}_{G/P_J} \# p^\#_2\calD^{-\lambda -\rho +w\rho +\mu }_{G/P_I}\modu)\]
 for $?=!, *$.
\end{defi}

For the composition of intertwining functors, the following holds. 
\begin{prop}~\\ \label{compose}
Let $I,J,K \subset \Pi$, $\mu_1 \in X^*(P_I), \mu_2 \in X^*(P_J)$ and $w_1, w_2 \in W$ satisfy $w_2K=J$, $w_1J=I$ and $\ell(w_1w_2)=\ell(w_1)+\ell(w_2)$. Then for $?=+$ and $?=!$, we have
\begin{align*}
R^{w_1w_2, \mu_1 + w_1\mu_2}_? \cong R^{w_2, \mu_2}_? \circ R^{w_1, \mu_1}_?.
\end{align*}
\end{prop}

\begin{proof}
Let $q_1 :\bO_{w_1w_2} \to \bO_{w_2} $ and $q_2: \bO_{w_1w_2} \to \bO_{w_1}$ be natural morphisms. 
We have the following diagram. 
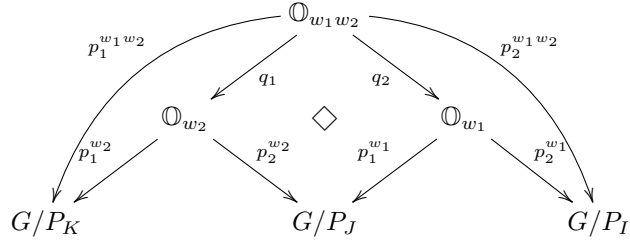
\[
\xymatrix{
          &         &  \bO_{w_1w_2}  \ar[dl]^{q_1} \ar[dr]_{q_2} \ar@/_25pt/[ddll]_{p^{w_1w_2}_1} \ar@/^25pt/[ddrr]^{p^{w_1w_2}_2} \ar@{}[dd]|{\rotatebox{45}{$\square$}} &    &    \\
          &  \bO_{w_2} \ar[dl]_{p^{w_2}_1} \ar[dr]^{p^{w_2}_2} &   & \bO_{w_1} \ar[dl]_{p^{w_1}_1} \ar[dr]^{p^{w_1}_2} & \\
  G/P_K  &   &  G/P_J  &   &   G/P_I \\
}
\]
The square is cartesian because of the equality $\ell(w_1w_2)=\ell(w_1)+\ell(w_2)$. 

We have $\det \Theta_{p^{w_1}_1} \cong \calL^{-\rho+w_1\rho}_{\bO_{w_1}}$ and $\det \Theta_{p^{w_2}_1} \cong \calL^{-\rho+w_2\rho}_{\bO_{w_2}}$. From this we obtain 
\[q^{*}_1(\det \Theta_{p^{w_2}_1} \otimes \calL^{\mu_2}_{\bO_{w_2}}) \otimes q^{*}_2(\det \Theta_{p^{w_1}_1} \otimes \calL^{\mu_1}_{\bO_{w_1}}) 
\cong \calL^{-\rho + w_1w_2\rho + w_1\mu_2 +\mu_1}_{\bO_{w_1w_2}} 
\cong \det \Theta_{p^{w_1w_2}_1} \otimes \calL^{w_1\mu_2 +\mu_1}_{\bO_{w_1w_2}}, 
\]
which by base change gives an isomorphism $K^{w_2, \mu_2}_? \ast K^{w_1, \mu_1}_? \cong K^{w_1w_2, \mu_1 + w_1\mu_2}_? $.
This isomorphism gives $R^{w_1w_2, \mu_1 + w_1\mu_2}_? \cong R^{w_2, \mu_2}_? \circ R^{w_1, \mu_1}_?$. Here we denote by $p_{12}, p_{23}, p_{13}$ the projection from $G/P_K \times G/P_J \times G/P_I$ to the product of two of the three factors and define the convolution of kernels by $K^{w_2, \mu_2}_? \ast K^{w_1, \mu_1}_? := p_{13+}(p^!_{12}(K^{w_2, \mu_2}_?) \totimes p^!_{23}(K^{w_1,\mu_1}_?))$
\end{proof}
This proposition and Proposition \ref{BH} due to Brink and Howlett allow to study the intertwining functor by the reduction to the maximal parabolic cases.

Intertwining functors for $w$ satisfying $wJ=I$ is an equivalence of categories. This is one of the main result in this paper. 

\begin{thm}\label{thmequiv}
The intertwining functors $R^{w, \mu}_+$ and $R^{w^{-1}, -w^{-1}\mu}_!$ are mutually inverse equivalences. 
\end{thm}
This theorem is a generalization of the result of Marastoni \cite[Theorem 1.1]{Mar13}.

We prove this theorem in two steps. First we prove this theorem for maximal parabolic case, i.e., the case when set $\Pi \setminus I$ consists of the unique element $\alpha$. 
In this case, $w$ satisfying Condition $\conda$ is the identity of $W$ or $w=w^{I}_0 w^{\Pi}_0$. We set $v:= w^{I}_0 w^{\Pi}_0$ and $J:=v^{-1}I \subset \Pi$. The $G$-orbit $\bO_v$ is open in $G/P_J \times G/P_I$. 

\begin{lem}\label{lemequiv}
Assume that $G$ is a simple algebraic group and $\Pi \setminus I =\{\alpha\}$. Let $v:=w^{I}_0 w^{\Pi}_0$ and $J:=v^{-1}I$. Let $\lambda \in (\fh/\fh_I)^*$ and $\mu \in X^*(P_I)$

Then the intertwining functors $R^{v,\mu}_+$ and $R^{v^{-1}, -v^{-1}\mu}_!$ are mutually inverse equivalences. 

\end{lem}

\begin{proof}
We shall prove the isomorphism $R^{v^{-1}, -v^{-1}\mu}_! \circ R^{v,\mu}_+ \cong {\id}$. The isomorphism $R^{v,\mu}_+ \circ R^{v^{-1}, -v^{-1}\mu}_! \cong {\id}$ is proved similarly. 

We consider following diagram. 
We denote by $p_1$ and $p_2$ (resp. $p'_1$ and $p'_2$, $p''_1$ and $p''_2$) the first and second projection from $G/P_J  \times G/P_I$ (resp. $G/P_I \times  G/P_J$, $G/P_I \times G/P_I$). We denote by $p_{12}, p_{23}, p_{13}$ the projection from $G/P_I \times G/P_J \times G/P_I$ to the product of two of three the factors. These morphisms are all smooth and proper morphisms.

\[
\xymatrix{
          &         &   G/P_I \times  G/P_I \ar@/_25pt/[llddd]_{p^{''}_1} \ar@/^25pt/[rrddd]^{p^{''}_2}  &        &    \\
          &         &  G/P_I \times G/P_J  \times  G/P_I \ar[u]_{p_{13}} \ar[dl]_{p_{12}} \ar[dr]^{p_{23}} &    &    \\
          &  G/P_I \times  G/P_J \ar[dl]_{p'_1} \ar[dr]^{p'_2} &   & G/P_J  \times G/P_I \ar[dl]_{p_1} \ar[dr]^{p_2} & \\
  G/P_I  &  \bO_{v^{-1}} \ar[l]_{ p^{{v}^{-1}}_1} \ar[r]^{p^{v^{-1}}_2 } \ar[u]_{j_{v^{-1}}} &  G/P_J  &  \bO_v \ar[l]_{p^{v}_1} \ar[r]^{p^{v}_2} \ar[u]_{j_v} &   G/P_I \\
}
\]

Using Lemma \ref{intker} the kernel which gives $R^{v^{-1}, -v^{-1}\mu}_! \circ R^{v,\mu}_+$ is calculated using base change as follows.

\begin{align}
R^{v^{-1}, -v^{-1}\mu}_! \circ R^{v,\mu}_+(\calM) & \cong p'_{1!}(K^{v^{-1}, -v^{-1}\mu}_! \totimes p'^!_2 \circ p_{1+} (K^{v,\mu}_+ \totimes p^!_2 (\calM)         )) \label{isom01}  \\
                                  & \cong p'_{1!}(K^{v^{-1}, -v^{-1}\mu}_! \totimes p_{12+} \circ p^!_{23}(K^{v,\mu}_+ \totimes (p_2 \circ p_{23})^!(\calM) )) \label{isom02} \\
    & \cong (p'_{1} \circ p_{12})_! (p^!_{12}(K^{v^{-1}, -v^{-1}\mu}_!) \totimes p^!_{23}(K^{v,\mu}_+) \totimes (p_2 \circ p_{23})^+ (\calM) ) \label{isom03} \\
       & \cong p''_{1+} (p_{13+}(p^!_{12}(K^{v^{-1}, -v^{-1}\mu}_!) \totimes p^!_{23}(K^{v,\mu}_+)) \totimes p''^!_2(\calM)) \label{isom04}
\end{align}
The isomorphism (\ref{isom02}) follows from the base change isomorphism (\ref{bc}). The isomorphism (\ref{isom03}) and (\ref{isom04}) follows from the projection formula (Proposition \ref{pullmon}, (ii)). We interchanged $*$ and $!$ for smooth and proper morphisms. 

Thus we see that the composition of intertwining functors are given by the convolution $K^{v^{-1}, -v^{-1}\mu}_! \ast K^{v,\mu}_+ := p_{13+}(p^!_{12}(K^{v^{-1}, -v^{-1}\mu}_!) \totimes p^!_{23}(K^{v,\mu}_+))$. Let $\Delta :G/P_I \to G/P_I \times G/P_I$ be the diagonal immersion. 
It is enough to show that there there is an isomorphism $K^{v^{-1}, -v^{-1}\mu}_! \ast K^{v,\mu}_+ \cong \Delta_+(\calO_{G/P_I \times G/P_I})$, since the latter kernel gives the identity functor.
To construct this isomorphism it is enough to prove the following two isomorphisms. 

\begin{align}
\left.(K^{v^{-1}, -v^{-1}\mu}_! \ast K^{v,\mu}_+)\right\rvert_{G/P_I \times G/P_I \setminus \Delta(G/P_I)} \cong 0 \label{eq1}\\
\Delta^!(K^{v^{-1}, -v^{-1}\mu}_! \ast K^{v,\mu}_+) \cong \calO_{G/P_I}[\dim G/P_I] \label{eq2}
\end{align}

\noindent{\it Proof of (\ref{eq1})}

Let $x_1,x_2$ be two distinct points of $G/P_I$. We define two open subsets of $G/P_J$ by $U_1:=p_1(p_2^{-1}(x_1) \cap \bO_v)$ and $U_2 :=p'_2(p'^{-1}_1(x_2) \cap \bO_{v^{-1}})==p_1(p_2^{-1}(x_2) \cap \bO_v)$. 
We denote by $s_1$ and $s_2$ the closed immersion of $U_1$ and $U_2$ into $\bO_v$ and $\bO_{v^{-1}}$, compatible with $p_{23}\circ \tilde x$ and $p_{12}\circ \tilde x$ and by $i_1$ and $i_2$ the open immersion of $U_1$ and $U_2$ into $G/P_J$. 

We consider following diagrams. 
We denote by $x$ the morphism $\{\star\} \to G/P_I \times G/P_I$ which sends $\star$ to $(x_1,x_2)$ and by $\tilde x$ the morphism $G/P_J \to G/P_I \times G/P_J \times G/P_I$ which sends $y \in G/P_J$ to $(x_1, y, x_2)$. 
\[
\xymatrix{
 y \ar@{|->}[d]  & G/P_J \ar@{}[l]|{\in} \ar[r]^{a_{G/P_J}} \ar[d]^{\tilde{x}} & \{\star\} \ar@{}[r]|{\ni} \ar[d]^{x} & \star \ar@{|->}[d]\\
(x_1, y ,x_2)  & G/P_I \times G/P_J \times G/P_I \ar@{}[l]|-{\in} \ar[r]^-{p_{13}} & G/P_I \times G/P_I  \ar@{}[r]|{\ni}&  (x_1, x_2)
}
\]
\[
\xymatrix{
 U_2 \ar[r]^{s_2} \ar[d]^{i_2}& \bO_{v^{-1}} \ar[d]^{j_{v^{-1}}} & \bO_{v} \ar[d]^{j_v}& U_1 \ar[d]^{i_1} \ar[l]_{s_1}\\
 G/P_J \ar[r]^-{p_{12} \circ \tilde{x}} & G/P_I \times G/P_J &  G/P_J \times G/P_I & G/P_J \ar[l]_-{p_{23} \circ \tilde{x}}
}
\]
We denote by $j_1$ and $j_2$ the open immersion of $U_1 \cap U_2$ into $U_1$ and $U_2$. 
\[
\xymatrix{
U_1 \cap U_2 \ar[r]^{j_1} \ar[d]_{j_2} & U_1 \ar[d]^{i_1} \\
U_2 \ar[r]_{i_2} & G/P_J
}
\]

It is enough to show the isomorphism $x^!(K^{w^{-1}}_! \ast K^{w}_+) \cong 0$. 
\begin{align}
x^!(K^{v^{-1}, -v^{-1}\mu}_! \ast K^{v,\mu}_+) & \cong x^!\circ p_{13+}( p^!_{12}(K^{v^{-1}, -v^{-1}\mu}_!) \totimes p^!_{23}(K^{v,\mu}_+))\label{isom11}\\
& \cong a_{G/P_J+}\circ \tilde{x}^! (p^!_{12}(K^{v^{-1}, -v^{-1}\mu}_!) \totimes p^!_{23}(K^{v,\mu}_+)) \label{isom12}\\
& \cong a_{G/P_J+}((p_{12} \circ \tilde{x})^!(K^{v^{-1}, -v^{-1}\mu}_!) \totimes (p_{23} \circ \tilde{x})^!(K^{v,\mu}_+) ) \label{isom13}\\
& \cong a_{G/P_J+}(i_{2!}\circ s^!_2( \det\Theta_{p^{\prime}_1} \totimes \calL^{-v^{-1}\mu}_{\bO_{v^{-1}}}) \totimes i_{1+}\circ s_{1}^!(\det\Theta_{p_1} \totimes \calL^{\mu}_{\bO_v})) \label{isom14}\\
& \cong a_{G/P_J+}(i_{2!}(\calO_{U_2}) \totimes i_{1+}(\calO_{U_1}) ) \label{isom15}\\
& \cong a_{G/P_J+} \circ i_{1+} \circ i_1^! \circ i_{2!}(\calO_{U_2}) \label{isom16}\\
& \cong a_{U_1+}(j_{1!}(\calO_{U_1 \cap U_2}))\label{isom17}
\end{align}
The isomorphism (\ref{isom12}) follows from the base change, (\ref{isom13}) follows from the fact that $\tilde x$ is a monoidal functor (Proposition \ref{pullmon} (i)) and (\ref{isom14}) follows from the base change. 
%The isomorphism (\ref{isom14}) uses Proposition \ref{nc} (noncharacteristic assumption is proved in \cite{Mar13}, with wrong proof)
The isomorphism (\ref{isom15}) is a consequence of the fact that the locally free sheaves $\Theta_{p'_1}$, $\Theta_{p_1}$ and invertible sheaves $\calL^{-v^{-1}\mu}_{\bO_{v^{-1}}}$ and $\calL^{\mu}_{\bO_v}$ are trivial on affine spaces $U_1$ and $U_2$. 
The isomorphism (\ref{isom16}) follows from the projection formula. 
The isomorphism (\ref{isom17}) follows from that we have $i^!_1 \cong i^+_1$ because $i_1$ is an open immersion, and that by the base change theorem we have $i^+_1 \circ i_{2!} \cong j_{1!} \circ j^+_2$.

The last term is a (non-twisted) regular holonomic D-module. We use the compatibility of six operations of D-modules on smooth algebraic varieties and six operations of constructible sheaves on associated complex manifolds under the de Rham functor $\DR(-):=\bR\calHom_{\calD_X}(\calO_X, -)$ (known as the Riemann-Hilbert correspondence). 

By the compatibility of the direct image functor and the de Rham functor \cite[\S 14.5.(1)]{Bor87}, we have 
\[a_{U_1+}j_{1!}\calO_{U_1 \cap U_2} \cong \bR\Gamma \DR(j_{1!}\calO_{U_1 \cap U_2}) \cong \bR\Gamma j_{1!}(\bC_{U_1 \cap U_2}).\]
Here for an algebraic variety $X$, we denote by $\bC_X$ the constant sheaf on associated complex manifold $X^{\textit{an}}$. Let $Z := U_1 \setminus (U_1 \cap U_2)$ be a closed subset of $U_1$ and $i_Z : Z \hookrightarrow U_1$ be the closed immersion. 

We have the following distinguished triangle of complexes of vector spaces.

\[ \bR\Gamma ( j_{1!}\bC_{U_1 \cap U_2}) \to \bR\Gamma( \bC_{U_1}) \to \bR\Gamma( i_{Z*}\bC_Z) \overset{+1}\to  \]
Since $U_1$ is an affine space the second term in this distinguished triangle is isomorphic to $\bC$ concentrated in degree 0. By the lemma below, the third term in this distinguished triangle is isomorphic to $\bC$ concentrated in degree 0 and the morphism is nonzero. From this we obtain $\bR\Gamma ( j_{1!}\bC_{U_1 \cap U_2}) \cong 0$.

\begin{lem}
Let $G$ be a semisimple algebraic group over $\bC$ and $P$ be a parabolic subgroup containing a Borel subgroup $B$. Let $C$ be the unique open $B$-orbit in $G/P$ and $Y$ be its complement. Then for any $g \in G$, the closed subvariety $C \cap gY$ of $C$ is contractible. 
\end{lem}

\begin{proof}
Since $C$ and $Y$ are $B$-stable, it is enough to consider the case when $g$ is a representative of some Weyl group element $w$. The subvariety $wY$ of $G/P_J$ is $T$-stable. Since $C$ contracts to a point by $\bG_m$-action induced by a dominant regular coweight of $T$, the closed $T$-stable subset $C \cap wY$ also contracts to a point. 
\end{proof}

\noindent{\it Proof of (\ref{eq2})}

We consider following diagrams. 

We denote by $\tau: G/P_I \times G/P_I \to G/P_I \times G/P_I$ and by $\tilde \tau : G/P_I \times G/P_J \to G/P_J \times G/P_I$ the permutation and by $\tilde \Delta$ and by $\tilde \Delta'$ the product of identity and $\Delta$. 
\[
\xymatrix{
& G/P_I \ar[d]^{\Delta}& \\
& G/P_I \times G/P_I \ar@(d,r)[]^{\tau} & \\
& G/P_I \times G/P_J \times G/P_I \ar[u]_{p_{13}} \\
G/P_I \times G/P_J \ar[ur]^{{\tilde \Delta}^{\prime}} \ar[uuur]^{p^{\prime}_1} \ar[rr]^{\tilde \tau}& & G/P_J \times G/P_I \ar[ul]_{\tilde \Delta} \ar[uuul]_{p_2}
}
\]

\[
\xymatrix{
\bO_v \ar[d]_{j_{v}} & \bO_{v^{-1}} \ar[l]^{\tilde{\tau} \vert_{\bO_{v^{-1}}}}_{\sim} \ar[d]^{j_{v^{-1}}} \\
 G/P_J \times G/P_I & G/P_I \times G/P_J \ar[l]^{\tilde \tau}_{\sim} 
}
\]

We have 
\begin{align}
\Delta^!(K^{v^{-1}, -v^{-1}\mu}_! \ast K^{v,\mu}_+) & \cong \Delta^! \circ p_{13+}( p^!_{12}(K^{v^{-1}, -v^{-1}\mu}_!) \totimes p^!_{23}(K^{v,\mu}_+) )\label{isom21}\\
& \cong p^{\prime}_{1+} \circ \widetilde{\Delta}'^{!}(p^!_{12}(K^{v^{-1}, -v^{-1}\mu}_!) \totimes p^!_{23}(K^{v, \mu}_+)) \label{isom22}\\
& \cong p^{\prime}_{1+} (\widetilde{\Delta}'^{!} \circ p^!_{12}(K^{v^{-1}, -v^{-1}\mu}_!) \totimes \widetilde{\Delta}'^{!} \circ p^!_{23}(K^{v,\mu}_+)) \label{isom23}\\
& \cong p^{\prime}_{1+}(K^{v^{-1}, -v^{-1}\mu}_! \totimes \tilde{\tau}^!(K^{v,\mu}_+)) \label{isom24}\\
& \cong p^{\prime}_{1+} \circ {j_{v^{-1}}}_+(\det\Theta_{p^{v^{-1}}_1} \totimes \calL^{-v^{-1}\mu}_{\bO_v^{-1}} \totimes  (\tilde\tau\vert_{\bO_{v^{-1}}})^*\det\Theta_{p^{v}_1} \totimes \calL^{v^{-1}\mu}_{\bO_v^{-1}}) \label{isom25}\\
& \cong \calO_{G/P_I}[\dim G/P_I] \label{isom26}
\end{align}

The isomorphism (\ref{isom22}) follows from the base change. The isomorphism (\ref{isom23}) follows from the fact that !-pullback is monoidal. The isomorphism (\ref{isom25}) follows from the projection formula. The isomorphism (\ref{isom26}) follows from the fact that $\det\Theta_{p^{v^{-1}}_1}$ and $(\tilde\tau\vert_{\bO_{v^{-1}}})^*\det\Theta_{p^{v}_1}$ are mutually inverse invertible sheaves, that $\calL^{-v^{-1}\mu}_{\bO_v^{-1}}$ and $\calL^{v^{-1}\mu}_{\bO_v^{-1}}$ are mutually inverse and the fact that $p^{v^{-1}}_1=p'_1 \circ j_{v^{-1}}$ is an affine space fibration. 
\end{proof}

\noindent{\it Proof of Theorem \ref{thmequiv}}\\
We shall prove the isomorphism $R^{w^{-1}, -w^{-1}\mu}_! \circ R^{w,\mu}_+ \cong {\id}$. The isomorphism $R^{w,\mu}_+ \circ R^{w^{-1}, -w^{-1}\mu}_! \cong {\id}$ is proved similarly. 

By Proposition \ref{BH} and Proposition \ref{compose}, it is enough to prove the theorem for $w=v[\alpha, J]:= w^{J \cup \{ \alpha \}}_0w^{J}_0$ for some $\alpha \in \Pi$ and $I=v[\alpha, J]J$.  
We assume this. 

We denote by $\alpha'$ the element of $\Pi$ such that $I \cup \{\alpha'\} = J \cup \{\alpha \}$. 
We have the following diagram. 
\[
 G/P_J \overset{p^w_1}\leftarrow \bO_{w} \overset{p^w_2}\rightarrow G/P_I 
\]

We consider the $\bar P_{J \cup \{\alpha \}}$-orbit of $eP_I$ and $eP_J$. These orbits are isomorphic to $L_{I \cup \{\alpha'\} }/P^{I \cup \{\alpha'\}}_{I} \times \bar U_{I \cup \{\alpha'\}}$ and $L_{J \cup \{\alpha\} }/P^{J \cup \{\alpha\}}_{J} \times \bar U_{J \cup \{\alpha\}}$ as algebraic varieties respectively. The pullback of these orbits coincide and isomorphic to $ \bO^{L_{I \cup \{\alpha'\} }}_{w} \times \bar U_{I \cup \{\alpha'\}}$, where $\bO^{L_{I \cup \{\alpha'\} }}_{w}$ is $\bO_{w}$ for $L_{I \cup \{\alpha'\} }$. 

By Lemma \ref{lemequiv}, we have an isomorphism $R^{w^{-1}, -w^{-1}\mu}_! \circ R^{w,\mu}_+(\calM) \cong \calM$ on $L_{I \cup \{\alpha'\} }/P^{I \cup \{\alpha'\}}_{I} \times \bar U_{I \cup \{\alpha'\}}$. Take any $x \in G/P_I$. Take the parabolic subgroup of $G$ corresponding to $x$ and take $B$, $\Pi$,$\dots$ compatibly. Then we have an isomorphism $R^{w^{-1}, -w^{-1}\mu}_! \circ R^{w,\mu}_+(\calM) \cong \calM$ near $x$. This completes the proof of the theorem.

%================================================================

\section{Intertwining functors and global sections}\label{sectsect}

\subsection{Global sections}

In this subsection we prove general properties of the global section functor $\bR\Gamma: D^-(\calD^\lambda_{G/P_I}\modu) \to D^-(\Gamma(\calD^\lambda_{G/P_I})\modu)$ and $\bL\Delta_I:D^-(\Gamma(\calD^\lambda_{G/P_I})\modu) \to D^-(\calD^\lambda_{G/P_I}\modu)$ in the case of partial flag varieties and for not necessarily antidominant $\lambda$ using results cited in \S \ref{TDOflag}. In this section we consider bounded above complexes because we do not know whether the algebra $\Gamma(\calD^\lambda_{G/P_I})$ is of finite global dimension.

\begin{lem} \label{lemsurj}
Assume that $\lambda$ is regular. Then $\psi^{\lambda}:{\rm U}^\lambda_I:=\calU(\fg)/I_{\fp_I}(\lambda-2\rho_{\fn_I}) \to \Gamma(\calD^\lambda_{G/P_I})$ is an isomorphism. 
\end{lem}

\begin{proof}
When $\lambda$ is antidominant, this is proved by Bien \cite[Proposition I.5.6]{Bie90}. This is also proved by combining Proposition \ref{asurj} and Proposition \ref{alphasurj}. 

By Proposition \ref{Dcohvan}, we have an isomorphism $\Gamma(\calD^{\lambda}_{G/P_I}) \cong \Gamma(\gr \calD^{\lambda}_{G/P_I}) \cong \Gamma (\calO_{T^*G/P_I}) $ as $G$-module for any $\lambda$. Hence the multiplicity of each finite dimensional $G$-module in $\Gamma(\calD^{\lambda}_{G/P_I})$ is finite and independent of $\lambda$. 

For general regular $\lambda$, pick $w\in W$ such that $I=wJ$ and $w^{-1}*\lambda$ is antidominant. 
Since $\psi^{\lambda}$ is injective, it is enough to show that both sides have the same finite multiplicity. 
By the result of Jantzen (Proposition \ref{sameideal}) and the equality $\rho -w\rho = \sum_{\alpha \in \Delta^+, w^{-1}\alpha < 0} \alpha = \rho_{\fn_I} - w \rho_{\fn_J}$, we see that equality  $I_{\fp_I}(\lambda-2\rho_{\fn_I}) = I_{\fp_J}(w^{-1}*\lambda-2\rho_{\fn_J})$ holds. 
Since $w^{-1}*\lambda$ is dominant, this implies $\calU(\fg)/I_{\fp_I}(\lambda-2\rho_{\fn_I}) \cong \calU(\fg)/I_{\fp_J}(w^{-1}*\lambda-2\rho_{\fn_J}) \cong \Gamma(\calO_{T^*G/P})$ as $G$-modules and hence they have the same finite multiplicity for any finite dimensional representation of $G$. Hence we see that $\calU(\fg)/I_{\fp_I}(\lambda-2\rho_{\fn_I})$ and $\Gamma(\calD^{\lambda}_{G/P_I})$ have the same finite multiplicity. 
\end{proof}

To prove a localization theorem for partial flag varieties, we need following two lemmas. 

\begin{lem}\label{counit}
The counit $\eta : \bR\Gamma \circ \bL\Delta_I \to \id$ is an isomorphism. 
\end{lem}

\begin{proof}
Let $M \in D^-(\Gamma(\calD^{\lambda}_{G/P_I})\modu)$.   
Take a free resolution $L$ of $M$. By Proposition \ref{Dcohvan} we have $\bR\Gamma \circ \bL\Delta_I (M) \cong \bR\Gamma \circ \Delta_I(L) \overset{\eta(L)}\longrightarrow L \cong M$. It is enough to show that $\Gamma \circ \Delta_I(L) \overset{\eta(L)}\longrightarrow L \cong M$ is an isomorphism. Since $L$ is a complex consisting of free $\Gamma(\calD^\lambda_{G/P_I})$-modules $\Delta_I(L)$ consists of free $\calD^{\lambda}_{G/P_I}$-modules. From this we deduce that $\eta(L)$ is an isomorphism. 
\end{proof}

\begin{lem}\label{faithful}
Assume that $\lambda$ is regular. 
Then the functor $\bR\Gamma$ is faithful. 
\end{lem}

\begin{proof}
We use the result of Kitchen (Proposition \ref{propKit}). 
The functor $\bR\Gamma(G/B,-)$ is an equivalence (Theorem \ref{BDequiv}). We can prove that the functor $p_{I+} \circ p^{!}_I$ has id as a direct summand in the same way as in \cite[Lemma 3.5.4]{BeiGinSoe96}. This implies that the functor $p^!_I$ is faithful. 
Since the composition functors $\bR\Gamma(G/B,-) \circ p^{!}_I \cong q^*_I \circ \bR\Gamma$ are faithful, we conclude that $\bR\Gamma: D^-(\calD^\lambda_{G/P_I}\modu) \to D^-(\Gamma(\calD^\lambda_{G/P_I})\modu)$ is faithful. 
\end{proof}

We now prove a localization theorem for $\calD^\lambda_{G/P_I}$-modules for not necessarily antidominant $\lambda$. 

\begin{prop}\label{propequiv}
Assume that $\lambda$ is regular. Then the functor $\bR\Gamma$ is an equivalence of categories.  
An inverse functor is given by $\bL\Delta_I$. 
\end{prop}

\begin{proof}
By Lemma \ref{counit}, $\eta$ is an isomorphism. We prove that $\epsilon$ is an isomorphism. 

Let $\calM \in D^b(\calD^{\lambda}_{G/P}\modu)$. 
Consider the distinguished triangle 
\[ \calM \stackrel{\epsilon(\calM)}{\longrightarrow} \bL\Delta_I\circ\bR\Gamma(\calM) \to C_{\epsilon(M)} \stackrel{+1}{\longrightarrow}, \]
where $C_{\epsilon(\calM)}$ is the mapping cone of the morphism $\epsilon(\calM)$. 

Apply $\bR\Gamma$ to this triangle.  We then obtain a distinguished triangle 

\[ \bR\Gamma(\calM) \stackrel{\bR\Gamma(\epsilon(\calM))}{\longrightarrow} \bR\Gamma\circ \bL\Delta_I\circ\bR\Gamma(\calM) \to \bR\Gamma(C_{\epsilon(M)}) \stackrel{+1}{\longrightarrow}. \]
Since $\bL\Delta_I$ is a left adjoint of $\bR\Gamma$, we have $\bR\Gamma(\epsilon(\calM))=\eta(\bR\Gamma(\calM)).$
Since $\eta$ is an isomorphism, we have $\bR\Gamma(C_{\epsilon(\calM)})=0.$ By Lemma \ref{faithful}, we have $C_{\epsilon(\calM)}=0$, which is equivalent to the statement that $\epsilon(\calM)$ is an isomorphism. 
\end{proof}

By Lemma \ref{lemsurj}, this proposition yields an equivalence $D^-(\calD^\lambda_{G/P}\modu) \cong D^-({\rm U}^\lambda_I\modu)$. 

\subsection{Global sections and intertwining functors}

In this subsection we study how the space of global sections behaves under intertwining functors. In this section we treat only $R^w_?$, i.e., set $\mu=0$. 

Let $\lambda \in (\fh/\fh_I)^*$. 
We have functors $\Gamma: \calD^\lambda_{G/P_I}\modu \to \Gamma(\calD^\lambda_{G/P_I})\modu$ and $\Gamma: \calD^{w^{-1}*\lambda}_{G/P_J}\modu \to \Gamma(\calD^{w^{-1}*\lambda}_{G/P_J})\modu$. The algebras $\Gamma(\calD^\lambda_{G/P_I})$ and $\Gamma(\calD^{w^{-1}*\lambda}_{G/P_J})$ are a priori not comparable. 
Here we consider their restriction to the quotient of enveloping algebra using $\psi^\lambda$ and $\psi^{w^{-1}*\lambda}$ in \S \ref{TDOflag}. We denote by $\Gamma^{\lambda}_I: \calD^\lambda_{G/P_I}\modu \to {\rm U}^\lambda_I\modu$ the composite $\psi^{\lambda*} \circ \Gamma$. 
As we have seen in the proof of Lemma \ref{lemsurj}, the codomains of functors $\bR\Gamma^\lambda_I$ and $\bR\Gamma^{w^{-1}*\lambda}_J \circ R^{w}_{+}$ coincide. 
The subject of this section is comparison of the functors $\bR\Gamma^{w^{-1}*\lambda}_J\circ R^w_+$, $\bR\Gamma^{w^{-1}*\lambda}_J \circ R^w_!$ and $\bR\Gamma^\lambda_I$.

We construct a morphism of functors $\bR\Gamma^\lambda_I \to \bR\Gamma^{w^{-1}*\lambda}_J \circ R^w_+$. 

Let $\calM \in D^b_{\textit hol}(\calD^{\lambda}_{G/P_I}\modu)$. 
\begin{align*}
R^{w}_{+}(\calM) & =  p_{1+}^{w}\big( \det(\Theta_{p^w_1}) \totimes p_2^{w!}(\calM)\big) \\
& = \bR p_{1*}^{w}\bigg(\big(p_1^{w*}\calD^{w^{-1}*\lambda, \textit{op}}_{G/P_J} \totimes \det(\Omega_{p^{w}_1})\big) \otimes^{\bL}_{\calD^{\lambda-\rho+w\rho}_{\bO_w}} \big(\det(\Theta_{p^w_1}) \totimes p^{w!}_{2}(\calM) \big) \bigg) \\
& \cong \bR p_{1*}^{w}\big( p_1^{w*}\calD^{w^{-1}*\lambda, \textit{op}}_{G/P_J} \otimes^{\bL}_{\calD^{\lambda}_{\bO_w}} p^{w!}_2(\calM) \big)
\end{align*}

Since $\calD^{w^{-1}*\lambda, \textit{op}}$ is a sheaf of rings, it has the section 1. Its pullback $p_1^{w!}\calD^{w^{-1}*\lambda, \textit{op}}_{G/P_J}$ also has a section induced from $1$. 
This section induces a morphism $p^{w!}_2\calM \to p^{w!}_{1}(\calD^{w^{-1}*\lambda, \textit{op}}) \otimes^{\bL}_{\calD^{\lambda}_{\bO_w}} p^{w!}_2\calM$.

We have the following sequence of morphisms of complex of vector spaces. 
\begin{align*}
 \bR\Gamma(\calM):= \bR\Gamma(G/P_I, \calM) \to \bR\Gamma(\bO_w, p^{w*}_2\calM) = \bR\Gamma(\bO_w, p^{w!}_2\calM) \cong \bR\Gamma(G/P_J, p^{w}_{1*} \circ p^{w!}_2\calM) \\
\to \bR\Gamma\big(G/P_J, p^w_{1*}((p^{w!}_{1}\calD^{w^{-1}*\lambda, \textit{op}})\otimes^{\bL}_{\calD^{\lambda}_{\bO_w}} p_2^{w!}\calM) \big) 
\cong \bR\Gamma(R^{w}_+\calM) 
\end{align*}
We denote by $I^w_+(\calM)$ the homomorphism given by the composition of these homomorphisms. 
Each of these maps is compatible with $\fg$-action. Thus we obtain a morphism of functors $I^w_+:\bR\Gamma^\lambda_I \to \bR\Gamma^{w^{-1}*\lambda}_J \circ R^w_+$. 
Since the functor $R^{w^{-1}}_!$ is inverse to $R^w_+$, we have $\bR\Gamma^{\lambda}_I\circ R^{w^{-1}}_! \to \bR\Gamma^{w^{-1}*\lambda}_J \circ R^{w}_+ \circ R^{w^{-1}}_! \cong \bR\Gamma^{w^{-1}*\lambda}_J$. 

Summarizing the above argument, we obtain the following proposition.

\begin{prop}\label{intmor}
We have natural morphism of functors $I^w_+: \bR\Gamma^{\lambda}_I \to \bR\Gamma^{w^{-1}*\lambda}_J \circ R^w_+$ 
and $I^w_!:\bR\Gamma^{\lambda}_I \circ R^{w^{-1}}_! \to \bR\Gamma^{w^{-1}*\lambda}_J$. 
\end{prop}

In the following we study when the morphism $I^w_+$ is an isomorphism. 

We first study the case where $\fp_I$ is a maximal parabolic subalgebra of $\fg$. 
The set $\Pi \setminus I$ consists of the unique element $\alpha$ and $(\fh/\fh_I)^*$ is a vector space of dimension one spanned by the fundamental weight $\varpi_\alpha$. 
In this case, $w$ is either identity of $W$ or $w=w^{I}_0 w^{\Pi}_0$. We set $v:= w^{I}_0 w^{\Pi}_0$ and $J:=v^{-1}I$. The $G$-orbit $\bO_v$ is open in $G/P_J \times G/P_I$. We have $\rho -v\rho=2\rho_{\fn_I}. $

\begin{lem}\label{RofD}
Assume that $G$ is a simple algebraic group and $\Pi \setminus I$ consists of one element. Let $v:=w^{I}_0 w^{\Pi}_0$ and $J:=v^{-1}I$. 

If $M^\fg_{\fp_J}(v^{-1}\lambda)$ is irreducible, then we have an isomorphism $\calD^{v^{-1}*\lambda}_{G/P_J} \cong R^v_+(\calD^{\lambda}_{G/P_I})$. 
\end{lem}

\begin{proof}
Since both are weakly $G$-equivariant $\calD^{v^{-1}*\lambda}_{G/P_J}$-modules, by Proposition \ref{wequivequiv} it is enough to check that their fibers are isomorphic to each other at the point $eP_J$. 

By Proposition \ref{fiberGVM}, we have an isomorphism $\calD^{v^{-1}*\lambda}_{P_J}(eP_J) \cong M^\fg_{\fp_J}(v^{-1}*\lambda-2\rho_{\fn_J})=M^\fg_{\fp_J}(v^{-1}\lambda)$. 

We consider the following diagram.
\[
\xymatrix{
eP_J \ar[d]_{i_e} & eP_J \times C_{v^{-1}} \ar[d]^{i_{ev^{-1}}} \ar[l]_p& \\
 G/P_J & \bO_v \ar[l]^{p^v_1} \ar[r]_{p^v_2} & G/P_I
}
\]
Taking a fiber at $eP_J$ is equivalent to applying $i^!_e$. 
We have
\begin{align}
R^{v}_+\calD^\lambda_{G/P_I}(eP_J) & \cong i^{!}_e p^{v+}_1 ( \det \Theta_{p^v_!} \totimes p^{v!}_2 \calD^\lambda_{G/P_I}) \label{isom31} \\
 & \cong p_+ i^!_{ev^{-1}} ( \det \Theta_{p^v_!} \totimes p^{v!}_2 \calD^\lambda_{G/P_I} )  \label{isom32}\\
 & \cong p_+ ( \det\Theta_{C_{v^{-1}}} \totimes i^!_{v^{-1}} \calD^\lambda_{G/P_I} ) \label{isom33} \\
 & = p_* ( \det\Omega_{C_{v^{-1}}} \otimes^\bL_{i^{\#}_{v^{-1}}\calD^{\lambda -\rho + v\rho}_{G/P_I}} (\det\Theta_{C_{v^{-1}}} \totimes i^!_{v^{-1}} \calD^\lambda_{G/P_I}) )  \label{isom34}  \\
 & \cong \Gamma(C_{v^{-1}}, \calO_{C_{v^{-1}}}). \label{isom35}
\end{align}
The isomorphism (\ref{isom32}) follows from the base change and the isomorphism (\ref{isom33}) follows from monoidal property of pullback. The isomorphism (\ref{isom35}) follows from the fact that $\det\Omega_{C_{v^{-1}}}$ and $\det\Theta_{C_{v^{-1}}}$ are mutually dual invertible sheaves. 

In the last term, the action of $\fg$ on $\calO_{C_{v^{-1}}}$ is via $i^{\#}_{v^{-1}}\calD^{\lambda}_{G/P_I}$. 
This $\fg$-module is $\fp_J$-finite. The section $1$ is of weight $v^{-1}w^I_0(\lambda -\rho ) -\rho = v^{-1}\lambda$ and the character of this module coincide with that of $M^{\fg}_{\fp_J}(v^{-1}\lambda) $. By the assumption  $M^\fg_{\fp_J}(v^{-1}\lambda)$ is irreducible and thus it is isomorphic to $M^\fg_{\fp_J}(v^{-1}\lambda)$. 
\end{proof}

Now we consider general $G$ and $I \subset \Pi$. 

Let $I, J \subset \Pi$ and $w \in W$ satisfy $I=wJ$. We fix $\alpha_1, \ldots, \alpha_r$ in Proposition \ref{BH} and let 
$I_0= I = v[\alpha_1, I_1]I_1,I_1 = v[\alpha_2, I_2]I_2, \ldots , I_{r-1} = v[\alpha_r, I_r]I_r , I_r=J$. 
By Proposition \ref{compose} we have an isomorphisms of functors $R^w_+ \cong R^{v[\alpha_r, I_r]}_+ \circ \cdots \circ R^{v[\alpha_1, I_1]}_+$. 

\begin{thm}\label{thmisom}
Let $\lambda \in (\fh/\fh_I)^*$. Let $\lambda_0 :=\lambda$ and $\lambda_i := v[\alpha_i, I_i]^{-1}*\lambda_{i-1}$. 
Assume that $\lambda$ is regular and for each $i$ the generalized Verma module $M^{\fl_{I_i \cup \{\alpha_i\}}}_{\fp^{I_i \cup \{\alpha_i\}}_{I_i}}({v[\alpha_i,I_i]^{-1}\lambda_{i-1}})$ is irreducible, then the morphism $I^w_+: \bR\Gamma^{\lambda}_I \to \bR\Gamma^{w^{-1}*\lambda}_J \circ R^w_{+}$ and $I^w_!:\bR\Gamma^{\lambda}_I \circ R^w_! \to \bR\Gamma^{w^{-1}*\lambda}_J$ are isomorphisms of functors. 
\end{thm}

Note that the each of the generalized Verma modules in the theorem is a tensor product of a generalize Verma module for some simple Lie algebra induced from a maximal parabolic subalgebra and a one dimensional representation. He, Kubo and Zierau give in \cite{HeKubZie??} a complete list of reducible parameters for such generalized Verma modules. Thus given $\lambda \in (\fh/\fh_I)^*$, we can determine whether $\lambda$ satisfies the assumption of the theorem by explicit computation. 

\begin{proof}
Since the functor $R^{w^{-1}}_!$ is an inverse of $R^w_+$, 
it is enough to show that $I^w_+: \bR\Gamma^{\lambda}_I \to \bR\Gamma^{w^{-1}*\lambda}_J \circ R^w_{+}$ is an isomorphism. 

We first prove that  $R^w_+\calD^{\lambda}_{G/P_I}$ is isomorphic to $\calD^{w^{-1}*\lambda}_{G/P_J}$. 

We use an argument similar to the one in Theorem \ref{thmequiv}.

Let $i$ be an integer satisfying $1 \leq i \leq r$. 
Over the open subvariety  $L_{I_i \cup \{\alpha_i\}}/P^{I_i \cup \{\alpha_i\}}_{I_i} \times \bar{U}_{I_{i} \cup \{ \alpha_i\}} $ of $G/P_{I_{i}}$, the diagram of the Radon transform $R^{v[\alpha_i,I_i]}_+$ is isomorphic to 
\begin{align*}
L_{I_i \cup \{\alpha_i\}}/P^{I_i \cup \{\alpha_i\}}_{I_i} \times \bar{U}_{I_i \cup \{\alpha_i\}} \leftarrow (p^{v[\alpha_i, I_i]}_1)^{-1}(L_{I_i \cup \{\alpha_i\}}/P^{I_i \cup \{\alpha_i\}}_{I_i} \times \bar{U}_{I_i \cup \{\alpha_i\}}) 
=\\
(p^{v[\alpha_i, I_i]}_2)^{-1}(L_{I_{i-1}\cup \{\alpha'_i\}}/P^{_{I_{i-1}\cup \{\alpha'_i\}}}_{I_{i-1}} \times \bar{U}_{I_{i-1}\cup \{\alpha'_i\}}) \rightarrow L_{I_{i-1}\cup \{\alpha'_i\}}/P^{_{I_{i-1}\cup \{\alpha'_i\}}}_{I_{i-1}} \times \bar{U}_{I_{i-1}\cup \{\alpha'_i\}}.
\end{align*}
Here $\alpha'_i$ is the simple root such that $\{\alpha'_i\} = ( I_i \cup \{\alpha_i \}) \setminus I_{i-1}$ holds. 

We have the following isomorphism of TDO's. 
\[\calD^{\lambda_{i-1}}_{G/P_{I_{i-1}}} \vert_{L_{I_{i-1}\cup \{\alpha'_i\}}/P^{_{I_{i-1}\cup \{\alpha'_i\}}}_{I_{i-1}} \times \bar{U}_{I_{i-1}\cup \{\alpha'_i\}}}  \cong \calD^{\langle \lambda_{i-1},\check\alpha'_i \rangle \varpi_{\alpha'_{i}}}_{ L_{I_{i-1} \cup \{\alpha'_{i}\}}/P^{I_{i-1} \cup \{\alpha'_{i}\}}_{I_{i-1}}} \boxtimes \calD_{\bar{U}_{I_{i-1} \cup \{\alpha'_i\}}}\]
\[\calD^{\lambda_i}_{G/P_{I_i}} \vert_{L_{I_i \cup \{\alpha_i\}}/P^{I_i \cup \{\alpha_i\}}_{I_i} \times \bar{U}_{I_i \cup \{\alpha_i\}}} \cong \calD^{\langle \lambda_i, \check\alpha_i \rangle\varpi_{\alpha_i}}_{L_{I_i \cup \{\alpha_i\}}/P^{I_i \cup \{\alpha_i\}}_{I_i}} \boxtimes \calD_{\bar{U}_{I_i \cup \{\alpha_i\}}}\]
Applying the intertwining functor, we obtain 
\[ R^{v[\alpha_i, I_i]}_+\calD^{\lambda_{i-1}}_{G/P_{I_{i-1}}}  \vert_{L_{I_{i-1}\cup \{\alpha'_i\}}/P^{_{I_{i-1}\cup \{\alpha'_i\}}}_{I_{i-1}} \times \bar{U}_{I_{i-1}\cup \{\alpha'_i\}}}  \cong  R^{v[\alpha_i, I_i]}_+\calD^{\langle \lambda_{i-1},\check\alpha'_i \rangle \varpi_{\alpha'_{i}}}_{ L_{I_{i-1} \cup \{\alpha'_{i}\}}/P^{I_{i-1} \cup \{\alpha'_{i}\}}_{I_{i-1}}} \boxtimes \calD_{\bar{U}_{I_{i-1} \cup \{\alpha'_i\}}}.\]
By Lemma \ref{RofD}, we have an isomorphism 
$R^{v[\alpha_i, I_i]}_+\calD^{\lambda_{i-1}}_{G/P_{I_{i-1}}} \cong \calD^{\lambda_{i}}_{G/P_{I_{i}}}$ on the open subset of $G/P_{I_i}$. By the weak equivariance of both sides and Proposition \ref{wequivequiv}, we see that they are isomorphic to each other on whole $G/P_{I_i}$.

Let $\calM \in D^b(\calD^{\lambda}_{G/P_I}\modu)$. 
Take a free resolution $M$ of $\bR\Gamma(\calM)$ in $D^-(\Gamma(\calD^{\lambda}_{G/P_I})\modu)$. 
Then by Proposition \ref{propequiv}, we have an isomorphism $\Delta_I(M) \cong \calM$ in $D^-(\calD^{\lambda}_{G/P_I}\modu)$. 
The morphism $I^w_+(\calD^\lambda_{G/P_I}) : \bR\Gamma^\lambda_I \calD^\lambda_{G/P_I} \to \bR\Gamma^{w^{-1}*\lambda}_J \circ R^w_{+}\calD^\lambda_{G/P_I} \cong \bR\Gamma^{w^{-1}*\lambda}_J\calD^{w^{-1}*\lambda}_{G/P_J}$ is an isomorphism by Proposition \ref{Dcohvan}, Lemma \ref{lemsurj} and Proposition \ref{sameideal}.
This implies that $I^w_+(\Delta_I(M))$ is an isomorphism. We conclude that $I^w_+(\calM) : \bR\Gamma^{\lambda}_I \calM \cong  \bR\Gamma^{\lambda}_I \circ \Delta_I(M) \to \bR\Gamma^{w^{-1}*\lambda}_J \circ R^w_{+} \circ\Delta_I(M) \cong \bR\Gamma^{w^{-1}*\lambda}_J \circ R^w_{+} \calM$ is an isomorphism. 
\end{proof}


\begin{thebibliography}{99}

\bibitem{ArkGai??}
Arkhipov, S; Gaitsgory, D. 
Localization and the long intertwining operator for representations of affine Kac-Moody algebras. preprint, 
http://www.math.harvard.edu/\~{}gaitsgde/GL/Arkh.pdf

\bibitem{BacKre15}
Backelin, E. ; Kremnizer, K. 
Singular localization of $\fg$-modules and applications to representation theory. 
J. Eur. Math. Soc. (JEMS) 17 (2015), no. 11, 2763-2787. 

\bibitem{BeiBer81}
Beilinson, A.; Bernstein, J. Localisation de $\mathfrak g$-modules. C. R. Acad. Sci. Paris S\'{e}r. I Math. 292 (1981), no. 1, 15-18.

\bibitem{BeiBer83}
Beilinson, A.; Bernstein, J. A generalization of Casselman's submodule theorem. Representation theory of reductive groups (Park City, Utah, 1982), 35-52, Progr. Math., 40, Birkh\"auser Boston, Boston, MA, 1983.

\bibitem{BeiBer93}
Beilinson, A.; Bernstein, J. A proof of Jantzen conjectures. I. M. Gelʹfand Seminar, 1-50, Adv. Soviet Math., 16, Part 1, Amer. Math. Soc., Providence, RI, 1993.

\bibitem{BeiGinSoe96}
Beilinson, A.; Ginzburg, V.; Soergel, W. Koszul duality patterns in representation theory. J. Amer. Math. Soc. 9 (1996), no. 2, 473-527. 

\bibitem{BerGelGel73}
Bernstein, I. N.; Gelfand, I. M.; Gelfand, S. I.
Schubert cells, and the cohomology of the spaces G/P.
Uspehi Mat. Nauk 28 (1973), no. 3(171), 3-26.

\bibitem{Bie90}
Bien, F. V. D-modules and spherical representations. Mathematical Notes, 39. Princeton University Press, Princeton, NJ, 1990. x+131 pp.

\bibitem{Bor87}
Borel, A.; Grivel, P.-P.; Kaup, B.; Haefliger, A.; Malgrange, B.; Ehlers, F.
Algebraic D-modules. 
Perspectives in Mathematics, 2. Academic Press, Inc., Boston, MA, 1987. xii+355 pp. 

\bibitem{BorBry82}
Borho, W. ; Brylinski, J.-L. Differential operators on homogeneous spaces. I. Irreducibility of the associated variety for annihilators of induced modules. Invent. Math. 69 (1982), no. 3, 437-476. 

\bibitem{BorBry89}
Borho, W. ; Brylinski, J.-L. 
Differential operators on homogeneous spaces. II. Relative enveloping algebras. 
Bull. Soc. Math. France 117 (1989), no. 2, 167-210. 

\bibitem{Bot57}
Bott, R. 
Homogeneous vector bundles. 
Ann. of Math. (2) 66 (1957), 203-248. 

\bibitem{BriHow99}
Brink, B.; Howlett, R. B. 
Normalizers of parabolic subgroups in Coxeter groups. 
Invent. Math. 136 (1999), no. 2, 323-351. 

\bibitem{CauDodKamX16}
Cautis, S.; Dodd, C.;Kamnitzer, J. 
Associated graded of Hodge modules and categorical $sl_2$ actions. 
arXiv:1603.07402

\bibitem{DAgSch96a}
D'Agnolo, A.; Schapira, P. 
Radon-Penrose transform for $\calD$-modules. 
J. Funct. Anal. 139 (1996), no. 2, 349-382. 

\bibitem{DAgSch96b}
D'Agnolo, A.; Schapira, P. 
Leray's quantization of projective duality. 
Duke Math. J. 84 (1996), no. 2, 453-496. 

\bibitem{HeKubZie??}
He, H.; Kubo, T.; and Zierau, R. 
On the reducibility of scalar generalized Verma modules associated to maximal parabolic subalgebras 
preprint, http://www.ms.u-tokyo.ac.jp/\~{}toskubo/HKZ.pdf

\bibitem{Jan77}
Jantzen, J. C.
Kontravariante Formen auf induzierten Darstellungen halbeinfacher Lie-Algebren. 
Math. Ann. 226 (1977), no. 1, 53-65. 

\bibitem{Jan83}
Jantzen, J. C. 
Einh\"ullende Algebren halbeinfacher Lie-Algebren. 
Ergebnisse der Mathematik und ihrer Grenzgebiete (3) Springer-Verlag, Berlin, 1983. ii+298 pp.

\bibitem{Kas89}
Kashiwara, M. Representation theory and D-modules on flag varieties. Orbites unipotentes et repr\'{e}sentations, III. Ast\'{e}risque No. 173-174 (1989), 9, 55-109. \\
http://www.kurims.kyoto-u.ac.jp/\~{}kenkyubu/kashiwara/rims622.pdf

\bibitem{KasTan96}
Kashiwara, M.; Tanisaki, T. Kazhdan-Lusztig conjecture for affine Lie algebras with negative level. II. Nonintegral case. Duke Math. J. 84 (1996), no. 3, 771-813.

\bibitem{KazLus80}
Kazhdan, D. ; Lusztig, G. 
Schubert varieties and Poincaré duality. Geometry of the Laplace operator (Proc. Sympos. Pure Math., Univ. Hawaii, Honolulu, Hawaii, 1979), pp. 185-203, 
Proc. Sympos. Pure Math., XXXVI, Amer. Math. Soc., Providence, R.I., 1980. 

\bibitem{Kit12}
Kitchen, S. N. Cohomology of standard modules on partial flag varieties. Represent. Theory 16 (2012), 317-344. 

\bibitem{LusVog83}
Lusztig, G. ; Vogan, D. A., Jr. 
Singularities of closures of K-orbits on flag manifolds. 
Invent. Math. 71 (1983), no. 2, 365-379. 

\bibitem{Mar98}
Marastoni, C. Grassmann duality for $\calD$-modules. Ann. Sci. \'{E}cole Norm. Sup. (4) 31 (1998), no. 4, 459-491. 

\bibitem{Mar13}
Marastoni, C. Integral geometry for $\calD$-modules on dual flag manifolds and generalized Verma modules. 
Math. Nachr. 286 (2013), no. 10, 992-1006. 

\bibitem{MarTan03}
Marastoni, C.; Tanisaki, T. 
Radon transforms for quasi-equivariant $\calD$-modules on generalized flag manifolds. 
Differential Geom. Appl. 18 (2003), no. 2, 147-176. 

\bibitem{Mil94}
Mili\v{c}i\'c, D. Localization and Representation Theory of Reductive Lie Groups. \\
https://www.math.utah.edu/\~{}milicic/Eprints/book.pdf

\bibitem{Soe89}
Soergel, W. 
Universelle versus relative Einh\"ullende: Eine geometrische Untersuchung von Quotienten von universellen Einh\"üllenden halbeinfacher Lie-Algebren. Math. Ann. 284 (1989), no. 2, 177-198.

\bibitem{Yun09}
Yun, Z. 
Weights of mixed tilting sheaves and geometric Ringel duality. 
Selecta Math. (N.S.) 14 (2009), no. 2, 299-320. 


\end{thebibliography}
\end{document}